\documentclass[11pt]{article}

\usepackage[T1]{fontenc}
\usepackage{textcomp}
\usepackage{newtxtext}

\makeatletter
\let\openbox\@undefined
\@ifundefined{cprime}{}{}
\@ifundefined{cdprime}{}{}
\@ifundefined{cyrdash}{}{}
\makeatother

\usepackage{geometry}
\usepackage{graphicx,color}
\usepackage[dvipsnames]{xcolor}
\usepackage{amsmath,amsthm, amsfonts, amssymb}
\usepackage{mathrsfs}
\usepackage{mathtools}
\usepackage{url}

\newtheorem{thm}{Theorem}[subsection]
\newtheorem{lmm}[thm]{Lemma}
\newtheorem{prp}[thm]{Proposition}
\newtheorem{dfn}[thm]{Definition}
\newtheorem{rmk}[thm]{Remark}
\newtheorem{cor}[thm]{Corollary}

\newtheorem{exm}[thm]{Example}

\usepackage{hyperref}
\hypersetup{
	bookmarks=true,
	bookmarksdepth=tocdepth,
	bookmarksnumbered=true,
	colorlinks=true,
	linkcolor=blue,
	citecolor=teal,
	urlcolor=PineGreen,
	pdftitle={},
	pdfsubject={},
	pdfauthor={},
	pdfkeywords={},
}

\title{A Generalized Fourier Transform and a Smooth Analogue of Dunkl Operators}
\author{Temma Aoyama}
\date{}

\begin{document}
	
	\maketitle
	\begin{abstract}
			We introduce a deformation of the Fourier transform on $\mathbb{R}^N$ arising from a representation-theoretic construction associated with $\widetilde{SL}(2,\mathbb{R}) \times O(N)$ that still admits an underlying degree-one operator structure. More precisely, we construct a generalized Fourier transform $\mathcal{F}_b$, a non-local deformation $H_b$ of the Laplacian $\Delta$, and operators $D_{b,n}$ deforming the partial derivatives $\frac{\partial}{\partial x_n}$. We show that the operators $D_{b,n}$ and $x_n$ are compatible with the $\widetilde{SL}(2,\mathbb{R})$-representation in a way parallel to the classical case: for each $n$, the space spanned by $x_n$ and $D_{b,n}$ carries the standard representation of $\widetilde{SL}(2,\mathbb{R})$; in particular, the generalized Fourier transform $\mathcal{F}_b$ interchanges $D_{b,n}$ and $x_n$, and the $\mathfrak{sl}_2$-triple is recovered from quadratic expressions in these operators. We also establish the inversion formula for $\mathcal{F}_b$ and give explicit formulas for both $\mathcal{F}_b$ and $D_{b,n}$. In particular, $\mathcal{F}_b$ admits an explicit integral kernel representation, and $D_{b,n}$ is expressed as the sum of a differential term and a spherical integral term. Our construction might be viewed as a continuous analogue of Dunkl theory, with $O(N)$ playing the role of a reflection group. 
	\end{abstract}
	
	\tableofcontents
	
	\section{Introduction}
	
	\subsection{Fourier analysis and Representation theory}
	We begin by reviewing some background material. The Fourier transform $\mathcal{F}$ on $L^{2}(\mathbb{R}^{N})$ can be expressed as follows \cite{MR974332}:
	\[
	\mathcal{F} = i^{\frac{N}{2}} \exp\!\left(\frac{\pi i}{4}(\Delta - |x|^{2})\right).
	\]
	This formula admits a natural interpretation in terms of a representation of $\widetilde{SL}(2,\mathbb{R}) \times O(N)$, where $\widetilde{SL}(2,\mathbb{R})$ denotes the universal covering of $SL(2,\mathbb{R})$. 
	
	\vspace{6pt}
	That is, the operators
	$\frac{i}{2}|x|^{2},\ \frac{i}{2}\Delta,\ E+\frac{N}{2}$ form an $O(N)$-invariant $\mathfrak{sl}_{2}$-triple, via the correspondence
	\[
	\frac{i}{2}|x|^{2}\leftrightarrow 
	\begin{pmatrix}
		0 & 1 \\
		0 & 0
	\end{pmatrix},\quad
	\frac{i}{2}\Delta\leftrightarrow 
	\begin{pmatrix}
		0 & 0 \\
		1 & 0
	\end{pmatrix},\quad
	E+\frac{N}{2}\leftrightarrow 
	\begin{pmatrix}
		1 & 0 \\
		0 & -1
	\end{pmatrix}.
	\]
	From this $\mathfrak{sl}_{2}$-triple, one obtains a representation of $\mathfrak{sl}_{2}(\mathbb{R})$, which lifts to a unitary representation of $\widetilde{SL}(2,\mathbb{R})$. Together with the natural action of $O(N)$, this yields a unitary representation
	\[\Omega\,:\,\widetilde{SL}(2,\mathbb{R})\times O(N)\curvearrowright L^{2}(\mathbb{R}^{N}). \]
	The Fourier transform $\mathcal{F}$ is realized as the action of a Weyl group element of $\widetilde{SL}(2,\mathbb{R})$ via $\Omega$:
	\[\mathcal{F}=i^{N/2}\Omega\,(e^{\frac{\pi }{2}\left(\begin{smallmatrix}
			0 & -1 \\
			1 & 0
		\end{smallmatrix}\right)}).\]
	The inversion formula can also be understood from this representation-theoretic viewpoint.
	
	We note that the space of smooth vectors of $\Omega$ coincides with $\mathcal{S}(\mathbb{R}^{N})$. This explains why the Fourier transform preserves the Schwartz space.
	
	\vspace{6pt}
	Moreover, this structure has an underlying degree-one part, generated by the operators
	$x_n$ and $\frac{\partial}{\partial x_n}$.
	The differential operators $\frac{\partial}{\partial x_n}$ commute among themselves:
	\[
	\left[\frac{\partial}{\partial x_m},\frac{\partial}{\partial x_n}\right]=0.
	\]
	The above $\mathfrak{sl}_2$-triple then has the following $O(N)$-invariant quadratic expression: 
	\[
	|x|^2=\sum_{n=1}^N x_n^2,\qquad
	E+\frac N2=\frac12\sum_{n=1}^N\left\{\frac{\partial}{\partial x_n},x_n\right\},\qquad
	\Delta=\sum_{n=1}^N\left(\frac{\partial}{\partial x_n}\right)^2.
	\]
	Furthermore, $\widetilde{SL}(2,\mathbb{R})$ acts on the real vector space
	$V_{n}:=\left\{x_{n},\, i\frac{\partial}{\partial x_{n}}\right\}_{\mathbb{R}}$ by the standard representation via
	\[
	(g,v)\mapsto \Omega(g)\circ v\circ \Omega(g)^{-1}
	\qquad
	(g\in\widetilde{SL}(2,\mathbb{R}),\ v\in V_n).
	\]
	This leads, in particular, to the relations
	\[
	\mathcal{F}\circ\frac{\partial}{\partial x_{n}}=ix_{n}\circ\mathcal{F},
	\qquad
	\mathcal{F}\circ x_{n}=i\frac{\partial}{\partial x_{n}}\circ \mathcal{F}.
	\]
	These can be interpreted as a manifestation of the Weyl group action exchanging weights.
	
	\vspace{6pt}
	$\Omega$ decomposes into irreducible components as
	\begin{equation}\label{blanchIntro}
		L^{2}\left(\mathbb{R}^{N}\right) \cong \sideset{}{^\oplus}{\sum}_{m=0}^{\infty}\pi_{\frac{N+2m -2}{2}}\boxtimes \mathcal{H}^{m}(\mathbb{R}^{N})
	\end{equation}
	where $\pi_{\lambda}$ is the lowest weight representation of lowest weight $\lambda+1$ with respect to the action of $-\frac{\Delta-|x|^{2}}{2}$, the infinitesimal generator of the $\widetilde{SO}(2)$-action under $\Omega$, and $\mathcal{H}^{m}(\mathbb{R}^{N})$ is the space of spherical harmonics of degree $m$. This structure provides a natural explanation for the above results.
	
	\vspace{12pt}
	For the realization of the representation $\Omega$ and the resulting representation-theoretic interpretation of the Fourier transform, we refer to Ben Sa\"id--Kobayashi--{\O}rsted \cite{MR2956043} in the special case \((k,a)=(0,2)\).
	
	\subsection{Main results}
	
	We now deform this structure. More precisely, we consider a deformation of the classical Fourier-analytic structure on $\mathbb{R}^N$ within the representation-theoretic framework of $\widetilde{SL}(2,\mathbb{R})\times O(N)$ that still admits an underlying degree-one structure parallel to the classical one.
	
	\vspace{6pt}
	
	We construct a generalized Fourier transform $\mathcal{F}_b$, a deformation $H_b$ of the Laplacian $\Delta$, and operators $D_{b,n}$ deforming the partial derivatives $\frac{\partial}{\partial x_{n}}$. 
	The operators $x_n$ and $D_{b,n}$ satisfy basic properties parallel to those in the classical case: the space
	$V_{b,n}:=\{x_n,\, iD_{b,n}\}_{\mathbb R}$
	carries the standard representation of $\widetilde{SL}(2,\mathbb{R})$ as in \eqref{stdintro}; in particular, the generalized Fourier transform $\mathcal{F}_b$ exchanges $x_n$ and $D_{b,n}$ as in \eqref{IWintro}, and the $\mathfrak{sl}_2$-triple is recovered from quadratic expressions in these operators as in \eqref{quadintro}.
	
	\vspace{6pt}
	Our construction proceeds as follows. We first define a non-local deformation $H_b$ of the Laplacian, and construct an $O(N)$-invariant $\mathfrak{sl}_{2}$-triple from $H_b$, thereby obtaining a $(\mathfrak g,\widetilde K)$-module $\omega_{b}$ for $(\mathfrak g,\widetilde K) =(\mathfrak{sl}_{2}(\mathbb{R}),\widetilde{SO}(2))$.
	By integrating $\omega_{b}$, we obtain a unitary representation $\Omega_b$ of  $\widetilde{SL}(2,\mathbb{R})\times O(N)$ on $L^2(\mathbb{R}^N,|x|^{2b}dx)$. 
	We then define the generalized Fourier transform $\mathcal{F}_b$ and the operators $D_{b,n}$ in terms of $\Omega_b$, and derive their basic properties representation-theoretically. We also compute their explicit formulas.
	
	\vspace{18pt}
	Let $b> -N/2$. We first construct a representation $\Omega_{b}$ of $\widetilde{SL}(2,\mathbb{R})\times O(N)$ on $L^{2}\!\left(\mathbb{R}^{N},|x|^{2b}dx\right)$. It decomposes as
			\[L^{2}\left(\mathbb{R}^{N},|x|^{2b}dx\right) \cong \sideset{}{^\oplus}{\sum}_{m=0}^{\infty}\pi_{b+\frac{N+2m -2}{2}}\boxtimes \mathcal{H}^{m}(\mathbb{R}^{N})\]
	where 
	\[
	\begin{aligned}
		\pi_{\lambda} \ &:\ \text{the lowest weight representation of }\widetilde{SL}(2,\mathbb{R})
		\text{ with lowest weight }\lambda+1,\\
		\mathcal{H}^{m}(\mathbb{R}^{N}) \ &:\ \text{the space of spherical harmonics of degree }m,
		\text{ an irreducible representation of }O(N).
	\end{aligned}
	\]
	That is, we consider the representation obtained by shifting all lowest weights from the classical case \eqref{blanchIntro} simultaneously by $b$.
	\vspace{12pt}
	
	This representation is constructed by deforming the Laplacian $\Delta$ to the non-local operator $H_{b}$, together with the $\mathfrak{sl}_{2}$-triple
	\[
	\frac{i}{2}|x|^{2}\leftrightarrow 
	\begin{pmatrix}
		0 & 1 \\
		0 & 0
	\end{pmatrix},\quad
	\frac{i}{2}H_{b}\leftrightarrow 
	\begin{pmatrix}
		0 & 0 \\
		1 & 0
	\end{pmatrix},\quad
	E+\frac{N+2b}{2}\leftrightarrow 
	\begin{pmatrix}
		1 & 0 \\
		0 & -1
	\end{pmatrix}.
	\]
	This $\mathfrak{sl}_{2}$-triple defines a $(\mathfrak g,\widetilde K)$-module $\omega_b$ for $(\mathfrak g,\widetilde K) =(\mathfrak{sl}_{2}(\mathbb{R}),\widetilde{SO}(2))$, which we then lift to a unitary representation $\Omega_b$.
	(See Definition~\ref{Hb} and Theorem~\ref{lift}.)
	
	 \vspace{6pt}
	 $H_{b}$ is a non-local operator on $L^{2}\left(\mathbb{R}^{N},|x|^{2b}dx\right)$ determined automatically by the above decomposition, together with the requirement that it commutes with the natural action of $O(N)$ and is compatible with the operators $|x|^{2}$ and $E$ (see Proposition~\ref{characterization}).
	 
	 We also note that the space of smooth vectors of $\Omega_{b}$ still coincides with $\mathcal{S}(\mathbb{R}^{N})$. 
	 
	 \vspace{24pt}
	 We define the generalized Fourier transform $\mathcal{F}_{b}$ via $\Omega_{b}$ as
	 \[\mathcal{F}_{b}=i^{b+N/2}\Omega_{b}\,(e^{\frac{\pi}{2}\left(\begin{smallmatrix}
	 		0 & -1 \\
	 		1 & 0
	 	\end{smallmatrix}\right)})= i^{b+\frac{N}{2}} \exp\!\left(\frac{\pi i}{4}(H_{b} - |x|^{2})\right).\]
	 It satisfies 
	 \[ \mathcal{F}_{b}^{2}f(x)=f(-x),\qquad \mathcal{F}_{b}\overline{\mathcal{F}}_{b}=\overline{\mathcal{F}}_{b}\mathcal{F}_{b}=1\]
	 (See Theorem~\ref{Fou}).
	 
	 \vspace{24pt}
	 To exhibit the underlying degree-one structure of $\Omega_{b}$, we also construct operators $D_{b,n}$ ($n=1,\dots,N$), which may be regarded as deformations of the partial derivatives. 
	 In parallel with the classical case, $\widetilde{SL}(2,\mathbb{R})$ acts on
	 $V_{b,n}:=\{x_n,\, iD_{b,n}\}_{\mathbb R}$
	 as the standard representation via
	 \begin{equation}\label{stdintro}
	 	(g,v) \mapsto \Omega_{b}(g)\circ v\circ\Omega_{b}(g)^{-1}\hspace{24pt} (g\in\widetilde{SL}(2,\mathbb{R}), v\in V_{b,n} )
	 \end{equation}
	 (See Corollary~\ref{lift2}). 
	 
	 Related to this, we derive the intertwining relations (See Theorem~\ref{fou1}):
	 \begin{equation}\label{IWintro}
	 \mathcal{F}_{b}\circ D_{b,n}=ix_n\circ\mathcal{F}_{b},
	 \qquad
	 \mathcal{F}_{b}\circ x_n=iD_{b,n}\circ\mathcal{F}_{b},
	 \end{equation}
	 the commutativity relations (See Corollary~\ref{relation}):
	 \[
	 [D_{b,m},D_{b,n}]=0,
	 \]
	 and the quadratic relations (See Corollary~\ref{relation}):
	 \begin{equation}\label{quadintro}
	 |x|^2=\sum_{n=1}^N x_n^2,\qquad
	 E+\frac{N+2b}{2}=\frac12\sum_{n=1}^N\{D_{b,n},x_n\},\qquad
	 H_b=\sum_{n=1}^N D_{b,n}^2.
	 \end{equation}
	 
	 In this sense, the operators $x_n$ and $D_{b,n}$ form the degree-one structure associated with the $\mathfrak{sl}_2$-triple governing $\mathcal{F}_b$. 
	
	\vspace{24pt}
	We also derive explicit formulas for $\mathcal{F}_{b}$ and $D_{b,n}$.
	
	The generalized Fourier transform $\mathcal{F}_b$ admits an explicit integral kernel representation.
	More precisely, for $f\in \mathcal{S}(\mathbb{R}^{N})$,
	\[\mathcal{F}_{b}f(x) = c_{b,N}\int_{\mathbb{R}^{N}} B_{b}(x,y)\,f(y) \,|y|^{2b}dy .\] 
	
	via the kernel
	\[
	B_{b}(x,y)
	=
	\frac{1}{B(b,N/2)}
	\int_0^1
	u^{b-1}(1-u)^{\frac{N}{2}-1}
	\mathcal{J}_b\bigl(u|x||y|\bigr)\,e^{-i(1-u)\langle x,y\rangle}\,du.
	\]
	
	where 
	$c_{b,N}=\frac{1}{(2\pi)^{b+N/2}} \frac{\Gamma(N/2)\pi^{b}}{\Gamma(b+N/2)}$ and $\mathcal{J}_{\nu}(w)
	=
	\sum_{m=0}^{\infty}
	\frac{(-1)^{m}(w/2)^{2m}}{(\nu+1)_m m!}$ is a normalized Bessel function
	(See Theorem~\ref{ker} and Proposition~\ref{summation}).
	
	\vspace{6pt}
	The operator $D_{b,n}$ admits the following explicit expression in terms of the partial derivative and integral over spheres. More precisely, for $f\in \mathcal{S}(\mathbb{R}^{N})$,
	\begin{align*}
		\MoveEqLeft D_{b,n}f(x) =\frac{\partial f}{\partial x_{n}}(x) + \frac{2b}{\mathrm{vol}(S^{N-1})}\int_{|x|=|y|} \frac{x_{n}-y_{n}}{\left| x- y \right|^{N}} (f(x)-f(y))\,dy & \\
		& =\frac{\partial f}{\partial x_{n}}(x) + \frac{b}{\mathrm{vol}(S^{N-1})} \int_{S^{N-1}} \xi_{n}\frac{f(x)-f(\sigma_{\xi}(x))}{\langle \xi,x\rangle}\,d\xi, 
	\end{align*}
	
	where $d\xi$ is the $O(N)$-invariant measure on $S^{N-1}$,
	and $dy$ is the corresponding $O(N)$-invariant measure on the sphere
	$\{y\in\mathbb{R}^N : |x|=|y|\}$.
	(See Theorem~\ref{explicit}). This might be viewed as an analogue of Dunkl operators  \cite{MR951883}, corresponding at least formally to the case where the reflection group is $O(N)$ and the root system is $\Phi=S^{N-1}$.
	
	\vspace{12pt}
	The following table summarizes the deformation considered in this paper.
	\begin{center}
		\begin{tabular}{c|c|l}
			Classical ($b=0$) & Deformed ($b > -N/2$) & Comment \\
			\hline
			$\partial_{x_n}$ & $D_{b,n}$ & deformation of partial derivative \\
			$\Delta$ & $H_b$ & non-local deformation of Laplacian \\
			$\mathcal{F}$ & $\mathcal{F}_b$ & generalized Fourier transform \\
			$\Omega$ & $\Omega_{b}$ & deformation of representation \\
			$L^{2}(\mathbb{R}^{N})$ & $L^{2}(\mathbb{R}^{N}, |x|^{2b}dx)$ & weighted $L^{2}$-space \\
			$\mathcal{S}(\mathbb{R}^{N})$ & $\mathcal{S}(\mathbb{R}^{N})$ & Schwartz space (Smooth vectors)
		\end{tabular}
	\end{center}
	
	\vspace{12pt}
	Our approach in Section~\ref{2} follows the method of \cite{MR2401813, MR2566988,MR2956043} developed in the analysis of minimal representations by Kobayashi--Mano \cite{MR2401813}, and subsequently extended by Ben Sa\"id--Kobayashi--{\O}rsted to the setting of \((k,a)\)-generalized Laguerre semigroups and \((k,a)\)-generalized Fourier analysis \cite{MR2566988,MR2956043}. Their work is based on representations of $\widetilde{SL}(2,\mathbb{R})$. The present study began with the aim of uncovering an additional degree-one structure associated with the $\widetilde{SL}(2,\mathbb{R})$-framework.

	\section{A generalized Fourier transform via representation theory}\label{2}
	In this section, we consider a one-parameter deformation of the Fourier transform.
	Our approach follows the method of \mbox{Ben Sa\"{i}d}--Kobayashi--{\O}rsted developed in  $(k,a)$-generalized Laguerre semigroup theory and $(k,a)$-generalized Fourier analysis \cite{MR2566988,MR2956043}.
	
	\vspace{10pt}
	
	\subsection{An orthogonal basis}\label{2.1}
	In this subsection, we construct a complete orthogonal basis of $L^{2}\left(\mathbb{R}^{N},|x|^{2b}dx\right)$.  This gives a basis of the space of $K$-finite vectors for the representation $\Omega_{b}$ of $\widetilde{SL}(2,\mathbb{R}) \times O(N)$ constructed in Subsection~\ref{2.2}.
	\vspace{6pt}
	
	Let $L_{\ell}^{(\nu)}(t) := \sum_{k=0}^{\ell}\frac{(-1)^{k}\Gamma(\nu+\ell +1)}{\Gamma(\nu + k +1) (\ell-k)!}\frac{t^{k}}{k!}$ \hspace{5pt} ($\ell\in\mathbb{Z}_{\ge 0}$, $\nu > -1$) be Laguerre polynomials.\par
	Let $\mathcal{H}^{m}(\mathbb{R}^{N})$ be the space of harmonic polynomials of degree $m$. Based on the facts that $\bigoplus_{m=0}^{\infty}\mathcal{H}^{m}(\mathbb{R}^{N})\hookrightarrow L^{2}(S^{N-1}) : p\mapsto p|_{S^{N-1}}$ is injective with dense image, and that when $m\neq m'$, $\mathcal{H}^{m}(\mathbb{R}^{N}) \perp_{L^{2}(S^{N-1})} \mathcal{H}^{m'}(\mathbb{R}^{N})$, we choose harmonic polynomials $p_{j}(x)$ \hspace{5pt} (if $N\ge 2$,  $j\in\mathbb{Z}_{\ge 0}$, and if $N=1$, $j=0,1$) forming a C.O.N.S. (complete orthonormal system) of $L^{2}(S^{N-1})$ and satisfying $\deg(p_{j})\le \deg(p_{j+1})$. We set $m_{j}:=\deg(p_{j})$.
	\vspace{18pt}
	
	Suppose $ b > -\frac{N}{2}$. For $\ell,j\in\mathbb{Z}_{\ge 0}$, we define the function $\Phi_{b,\ell,j}(x)$ as follows:
	\[\Phi_{b,\ell,j}(x) := e^{-\frac{1}{2}|x|^2}\,L^{(b+\lambda_{N,m_{j}})}_{\ell}\bigl(|x|^2\bigr) \,p_{j}(x),\]
	where $\lambda_{N,m}:=\frac{N-2}{2}+m$. We will also use the following notation, depending on the context:\\
	For $p(x) \in \mathcal{H}^{m}(\mathbb{R}^{N})$,
	\[\Phi_{b,\ell,p}(x):= e^{-\frac{1}{2}|x|^2}\,L^{(b+\lambda_{N,m})}_{\ell}\bigl(|x|^2\bigr) \,p(x).\]
	
	\vspace{3pt}
	$\Phi_{b,\ell,p}(x)$ and $\Phi_{b,\ell,j}(x)$ will later be shown to diagonalize the generalized harmonic oscillator introduced below; see Proposition~\ref{har}. Moreover, the space $W_{b,alg}$ spanned by them will serve as the space of $K$-finite vectors for a representation $\Omega_{b}$ constructed in Theorem~\ref{lift}; see Proposition~\ref{act1}.
	
	\vspace{6pt}
	\begin{dfn}[The space $W_{b,alg}$]\label{Walg}
		We define
		\[
		W_{b,alg}:=\bigoplus_{\ell,j}\mathbb C \Phi_{b,\ell,j}(x)
		\subset L^2(\mathbb R^N,|x|^{2b}dx).
		\]
	\end{dfn}
	
	\vspace{18pt}
	The following proposition shows that the family $\{\Phi_{b,\ell,j}(x)\}_{\ell,j}$ forms the C.O.S. (complete orthogonal system) of $L^{2}(\mathbb{R}^{N},|x|^{2b}dx)$.
	
	\begin{prp}[Orthogonal basis of $L^{2}(\mathbb{R}^{N},|x|^{2b}dx)$]\label{orth}
		
		\begin{enumerate}
			\item 
			\begin{align*} 
				\int_{\mathbb{R}^{N}}\Phi_{b, \ell,j}(x)\,	\overline{\Phi_{b, \ell',j'}(x)}\, |x|^{2b}dx= \delta_{\ell,\ell'} \delta_{j,j'}\frac{\Gamma(\ell+b+\lambda_{N,m_{j}}+1)}{2\Gamma(\ell+1)}.
			\end{align*}
			\item Suppose $ b > -\frac{N}{2}$. $W_{b,alg}$ is a dense subspace of $ L^{2}(\mathbb{R}^{N}, |x|^{2b}dx)$.
		\end{enumerate}
	\end{prp}
	\begin{proof}
		
		\begin{enumerate}
			\item 
			\begin{align*} 
				\MoveEqLeft\int_{\mathbb{R}^{N}}\Phi_{b,\ell,j}(x)\, \overline{\Phi_{b,\ell',j'}(x)}\, |x|^{2b}dx &\\
				&= \int_{0}^{\infty} L^{(b+\lambda_{N,m_{j}})}_{\ell}\bigl(r^2\bigr) L^{(b+\lambda_{N,m_{j'}})}_{\ell'}\bigl(r^2\bigr)\,r^{2b+m_{j}+m_{j'}+N-1}e^{-r^2}dr \int_{S^{N-1}} p_{j}(\omega)\,\overline{p_{j'}(\omega)}\,d\omega &\\
				&= \frac{1}{2} \int_{0}^{\infty} L^{(b+\lambda_{N,m_{j}})}_{\ell}(t) L^{(b+\lambda_{N,m_{j}})}_{\ell'}(t)\,t^{b+\lambda_{N,m_{j}}}e^{-t}dt \times \delta_{j,j'} \\
				&= \delta_{\ell,\ell'} \delta_{j,j'}\frac{\Gamma(\ell+b+\lambda_{N,m_{j}}+1)}{2\Gamma(\ell+1)}.
			\end{align*}
			
			In the last equality, we used $\int_{0}^{\infty} L^{(\nu)}_{\ell}(t) L^{(\nu)}_{\ell'}(t)\,t^{\nu}e^{-t}dt =  \delta_{\ell,\ell'} \frac{\Gamma(\ell+\nu+1)}{\ell!}$ (the orthogonality of Laguerre polynomials).
			\item When $\alpha >-1$, $\operatorname{span}_{\mathbb C}\{ e^{-x/2}x^{\alpha/2}x^{n} \mid n=0,1, \dots\}$ is dense in $L^2(0,\infty)$ \cite[Theorem 5.7.1]{MR372517}. 
			Thus, when $b>-\frac{N}{2}$, the space $\operatorname{span}_{\mathbb C}\{ \Phi_{b,\ell,p}(x) \mid \ell \in \mathbb{Z}_{\ge 0},\,p(x)\in\mathcal{H}^{m}(\mathbb{R}^{N}) \} $ is dense in  $L^{2}(\mathbb{R}_{>0}, r^{N+2b-1}dr)\otimes \mathcal{H}^{m}(\mathbb{R}^{N})|_{S^{N-1}}$ for all $m\ge0$, where $r=|x|$. 
			Since $\bigoplus_{m=0}^{\infty}\mathcal{H}^{m}(\mathbb{R}^{N})|_{S^{N-1}}$ is dense in $L^{2}(S^{N-1})$, $\bigoplus_{m=0}^{\infty} L^{2}(\mathbb{R}_{>0}, r^{N+2b-1}dr)\otimes \mathcal{H}^{m}(\mathbb{R}^{N})|_{S^{N-1}}$ is dense in  $L^{2}(\mathbb{R}^{N}, |x|^{2b}dx)$. Hence, $W_{b,alg}$ is dense in $L^{2}(\mathbb{R}^{N}, |x|^{2b}dx)$.
		\end{enumerate}
		
	\end{proof}

	\vspace{12pt}
	
	\subsection{A unitary representation of $\widetilde{SL}(2,\mathbb{R}) \times O(N)$}\label{2.2}
	In this subsection, we construct a unitary representation of $\widetilde{SL}(2,\mathbb{R}) \times O(N)$. 
	
	We first introduce the operator $H_b$ in Definition~\ref{Hb} and show in Proposition~\ref{sl} that it gives rise to an $O(N)$-invariant $\mathfrak{sl}_2$-triple on $L^2(\mathbb{R}^N,|x|^{2b}dx)$. We then use this triple to construct a $(\mathfrak g,\widetilde K)$-module $\omega_b$  in Definition~\ref{gK} on the space $W_{b,alg}$ introduced in Definition~\ref{Walg}, and finally lift it to a unitary representation $\Omega_b$ of $\widetilde{SL}(2,\mathbb{R})\times O(N)$ on $L^2(\mathbb{R}^N, |x|^{2b}dx)$ in Theorem~\ref{lift}.

	\vspace{12pt}
	
	We begin by introducing the operator $H_b$, which plays the role of a deformed Laplacian in our construction.
	\begin{dfn}[The operator $H_{b}$]\label{Hb}
	We define an operator $H_{b}$ on $L^2(\mathbb{R}^N, |x|^{2b}dx)$ with domain $W_{b,alg}$ by
	\[
	H_b := \Delta + \frac{2b}{|x|^2}\mathcal{R},
	\]
	
	where $\mathcal R$ is defined on functions of the form $f(|x|^2)p(x)$ (which include the basis elements of $W_{b,alg}$) by
	\[
	\mathcal R\bigl(f(|x|^2)p(x)\bigr)
	=
	E\bigl(f(|x|^2)\bigr)p(x),
	\]
	with $E=\sum_{j=1}^N x_j \partial_{x_j}$ the Euler operator. Here $W_{b,alg}$ is the space defined in Definition~\ref{Walg}.
\end{dfn}
	
	\begin{rmk}
		The operator $H_b$ is uniquely determined by certain representation-theoretic conditions, namely, compatibility with the $O(N)$-action, with the operators $E$ and $|x|^2$, and with an irreducible decomposition of a specified form; these motivate its introduction. See Proposition~\ref{characterization}.
	\end{rmk}
	
	\begin{rmk}[An alternative expression for $\mathcal R$]
		\[\mathcal{R}=E+\frac{N-2}{2}-\left(\left(\frac{N-2}{2}\right)^2-\Delta_{S^{N-1}}\right)^{1/2}\]
		in an appropriate sense, since
		$\Delta_{S^{N-1}}p(\omega)=-m(m+N-2)p(\omega)$ for $p\in\mathcal H^m(\mathbb R^N)$.
		In particular, $\mathcal R$ is non-local. We note that the operator $\left(\left(\frac{N-2}{2}\right)^2-\Delta_{S^{N-1}}\right)^{1/2}$ also appears in \cite{MR2020550}.
	\end{rmk}
	
	\vspace{6pt}
	The next proposition shows that $H_b-|x|^2$ is diagonalized by the basis introduced in Subsection~\ref{2.1}. This will be the starting point for the representation-theoretic construction.
	\begin{prp}[The generalized harmonic oscillator]\label{har}
		For $p(x) \in \mathcal{H}^{m}(\mathbb{R}^{N})$,
		$$ \,\frac{1}{2}\left(H_{b}-|x|^2\right)\, \Phi_{b,\ell,p}(x) = -(b+\lambda_{N,m}+2\ell +1)\,\Phi_{b,\ell,p}(x).\,$$
	\end{prp}
	\begin{proof}
		The claim follows from a computation in polar coordinates together with 
		the Laguerre differential equations $t\frac{d^2L^{(\alpha)}_{\ell}(t)}{dt^2}+ (\alpha +1 -t)\frac{dL^{(\alpha)}_{\ell}(t)}{dt} + \ell L^{(\alpha)}_{\ell}(t)= 0$.
	\end{proof}

	\begin{cor}[Essential self-adjointness]\label{adjoint}
		$\frac{1}{2}\left(H_{b}-|x|^2\right)$ is essentially self-adjoint on $L^{2}(\mathbb{R}^{N},|x|^{2b}dx)$. In particular, it generates a one-parameter unitary group
		$e^{\frac{it}{2}\left(H_{b}-|x|^2\right)}$.
	\end{cor}
	
	\begin{proof}
		By Proposition~\ref{orth}, eigenfunctions of $\frac{1}{2}\left(H_{b}-|x|^2\right)$ form C.O.N.S. of $L^{2}(\mathbb{R}^{N},|x|^{2b}dx)$. This proves the claim.
	\end{proof}
	
	\begin{rmk}[$\widetilde{K}$-action]
		The one-parameter group $e^{\frac{it}{2}\left(H_{b}-|x|^2\right)}$ will define the $\widetilde{K}=\widetilde{SO}(2)$-action on the $(\mathfrak g,\widetilde K)$-module $\omega_b$ introduced in Definition~\ref{gK}.
	\end{rmk}
	
	\begin{rmk}[Discrete spectrum]
		In particular, $H_{b}-|x|^2$ has purely discrete spectrum.
	\end{rmk}
	
	\vspace{24pt}
	We now reinterpret these operators in terms of an $\mathfrak{sl}_2$-triple and apply representation theory.
	
	\begin{prp}[The $\mathfrak{sl}_2$-triple associated with $H_b$]\label{sl}
		The operators $\frac{i}{2}|x|^2$, $\frac{i}{2}H_{b}$, $E+\frac{N+2b}{2}$ act on $W_{b,alg}$ and form an $\mathfrak{sl}_{2}$-triple. That is, 
		\[ \left[E+\frac{N+2b}{2},\frac{i}{2}|x|^2\right] = i|x|^2 \]
		\[ \left[E+\frac{N+2b}{2},\frac{i}{2}H_{b} \right] = -iH_{b} \]
		\[ \left[\frac{i}{2}|x|^2, \frac{i}{2}H_{b}\right] = E+\frac{N+2b}{2} \]
		hold, where $E:=\sum_{k=1}^{N}x_{k}\frac{\partial}{\partial x_{k}}$ is the Euler operator.
	\end{prp}
	\begin{proof}
		By the equalities $\bigl[ \mathcal{R}, |x|^{2}\bigr]=2|x|^{2}$ and $\bigl[ \mathcal{R}, E\bigr]=0$, the claim follows.
	\end{proof}
	
	\begin{rmk}\label{correspond}
		Based on Proposition~\ref{sl}, we identify $\mathfrak{sl}_{2}(\mathbb{R})$ with the operators $\frac{i}{2}|x|^{2},\ \frac{i}{2}H_{b},\ E+\frac{N+2b}{2}$ via the correspondence
		\[
		\frac{i}{2}|x|^{2}\leftrightarrow 
		\begin{pmatrix}
			0 & 1 \\
			0 & 0
		\end{pmatrix},\quad
		\frac{i}{2}H_{b}\leftrightarrow 
		\begin{pmatrix}
			0 & 0 \\
			1 & 0
		\end{pmatrix},\quad
		E+\frac{N+2b}{2}\leftrightarrow 
		\begin{pmatrix}
			1 & 0 \\
			0 & -1
		\end{pmatrix}.
		\]
	\end{rmk}
	
	\vspace{4pt}
	\begin{dfn}[The $(\mathfrak g,\widetilde K)$-module $\omega_b$]\label{gK}
		Let $\mathfrak g=\mathfrak{sl}_2(\mathbb R)$ and $\widetilde K=\widetilde{SO}(2)$, the universal covering of $SO(2)$.
		By Corollary~\ref{adjoint} and Proposition~\ref{sl}, the operators
		\[
		\frac{i}{2}|x|^2,\qquad \frac{i}{2}H_b,\qquad E+\frac{N+2b}{2}
		\]
		define a $(\mathfrak g,\widetilde K)$-module structure on $W_{b,alg}$. Here $W_{b,alg}$ is the space defined in Definition~\ref{Walg}.
		We denote this $(\mathfrak g,\widetilde K)$-module by $\omega_b$. 
	\end{dfn}
	
	The action of $\omega_b$ commutes with the natural action of $O(N)$ on $W_{b,alg}\subset L^2(\mathbb R^N,|x|^{2b}dx)$.
	
	\vspace{10pt}
	The action of $\mathfrak{sl}_{2}(\mathbb{R})$ on $W_{b,alg}$ via $\omega_{b}$ is described as follows:
	\begin{prp}[$\mathfrak{sl}_2$-action on $W_{b,alg}$]\label{act1}
		For $p(x) \in \mathcal{H}^{m}(\mathbb{R}^{N})$,
		\begin{align*}
			\MoveEqLeft \frac{H_{b}-|x|^2}{2}\, \Phi_{b,\ell,p}(x) = -(b+\lambda_{N,m}+2\ell+1)\Phi_{b,\ell,p}(x) \\
			\MoveEqLeft \left(E+ \frac{N+2b}{2} - \frac{H_{b}+|x|^{2}}{2}\right)\,  \Phi_{b,\ell,p}(x) =  2(\ell+1)\, \Phi_{b,\ell+1,p}(x)\\
			\MoveEqLeft \left(E+ \frac{N+2b}{2} + \frac{H_{b}+|x|^{2}}{2}\right)\,  \Phi_{b,\ell,p}(x) =  -2(b+\lambda_{N,m}+\ell)\Phi_{b,\ell-1,p}(x)
		\end{align*}
	\end{prp}
	\begin{proof}
		These follow from the following identities for Laguerre polynomials:
		\begin{align*}
			\MoveEqLeft t\frac{d^2L^{(\alpha)}_{\ell}(t)}{dt^2}+ (\alpha +1 -t)\frac{dL^{(\alpha)}_{\ell}(t)}{dt} +\ell L^{(\alpha)}_{\ell}(t)= 0 \\
			\MoveEqLeft tL^{(\alpha)}_{\ell}(t) = -(\ell+1)L^{(\alpha)}_{\ell+1}(t) +(2\ell+\alpha +1)L^{(\alpha)}_{\ell}(t)-(\ell+\alpha)L^{(\alpha)}_{\ell-1}(t) \\
			\MoveEqLeft t\frac{dL^{(\alpha)}_{\ell}(t)}{dt} = \ell L^{(\alpha)}_{\ell}(t)-(\ell+\alpha)L^{(\alpha)}_{\ell-1}(t).
		\end{align*}
	\end{proof}
	
	\begin{rmk}\label{Cayley}
		$\mathbf{h}=-\frac{H_{b}-|x|^2}{2},\quad \mathbf{e}=\frac{1}{2}\left(E+ \frac{N+2b}{2} - \frac{H_{b}+|x|^{2}}{2}\right),\quad \mathbf{f}=\frac{1}{2}\left(E+ \frac{N+2b}{2} + \frac{H_{b}+|x|^{2}}{2}\right)$ form an $\mathfrak{sl}_{2}$-triple. That is,
		\begin{align*}
				\left[\mathbf{h},\mathbf{e}\right] &= 2\mathbf{e} \\
			 \left[\mathbf{h},\mathbf{f} \right] &= -2\mathbf{f} \\
			\left[\mathbf{e}, \mathbf{f}\right] &= \mathbf{h}. 
		\end{align*}
		Based on Remark~\ref{correspond}, there is a correspondence $\mathbf{h}\leftrightarrow \bigl(\begin{smallmatrix}
			0 & -i \\
			i & 0
		\end{smallmatrix}\bigr),\, \mathbf{e}\leftrightarrow \frac{1}{2}\bigl(\begin{smallmatrix}
		1 & i \\
		i & -1
		\end{smallmatrix}\bigr),\, \mathbf{f}\leftrightarrow \frac{1}{2}\bigl(\begin{smallmatrix}
		1 & -i \\
		-i & -1
		\end{smallmatrix}\bigr)$. We note that these matrices are obtained from $\bigl(\begin{smallmatrix}
		1 & 0 \\
		0 & -1
		\end{smallmatrix}\bigr),\,\bigl(\begin{smallmatrix}
			0 & 1 \\
			0 & 0
		\end{smallmatrix}\bigr),\,\bigl(\begin{smallmatrix}
			0 & 0 \\
			1 & 0
		\end{smallmatrix}\bigr)$ by the Cayley transform. 
	\end{rmk}
	
	\vspace{10pt}
	We now arrive at the main point of this subsection: the $(\mathfrak g,\widetilde K)$-module $\omega_b$ integrates to a unitary representation of $\widetilde{SL}(2,\mathbb{R})$
	(the universal covering of $SL(2,\mathbb{R})$). The proof follows that in \mbox{Ben Sa\"{i}d}--Kobayashi--{\O}rsted \cite[Section 3.6]{MR2956043}.
	\begin{thm}[Lifting to a unitary representation]\label{lift}
		The $(\mathfrak{g},\widetilde{K})$-module $\omega_{b}$ defined in Definition~\ref{gK} lifts to a unitary representation of $\widetilde{SL}(2,\mathbb{R})$. 
		More precisely, there exists a unique unitary representation $\Omega_{b}$ of $\widetilde{SL}(2,\mathbb{R})$ on $L^{2}(\mathbb{R}^{N},|x|^{2b}dx)$ such that for each $X\in\mathfrak{sl}_2(\mathbb R)$ and $v\in W_{b,alg}$,
		$$ \left.\frac{d}{dt}\right|_{t=0}\Omega_{b}(e^{tX}) v= \omega_{b}(X)v$$
		and for each $k\in\widetilde{K}$ and $v\in W_{b,alg}$,
		$$ \Omega_{b}(k)v = \omega_{b}(k)v.$$
	\end{thm}
	\begin{proof}
		By Proposition~\ref{act1}, $W_{b,alg} $ decomposes as a representation of $\mathfrak{sl}_{2}(\mathbb{R})\times O(N)$:
		$$W_{b,alg} \cong \sideset{}{^\oplus}{\sum}_{m=0}^{\infty}\pi_{K,b+\lambda_{N,m}}\boxtimes \mathcal{H}^{m}(\mathbb{R}^{N}).$$
		Here, $\pi_{K,\lambda}$ is the lowest weight representation of $\mathfrak{sl}_{2}(\mathbb{R})$ with weight $\lambda +1$. Here the lowest weight is determined from the spectral decomposition of $-\frac{1}{2}(H_{b}-|x|^2)$.  $\mathcal{H}^{m}(\mathbb{R}^{N})$ is the space of spherical harmonics of degree $m$. \par
		
		\vspace{5pt}
		By the facts that ``For a real $\lambda$ with $\lambda > -1$, there exists a unique unitary representation, denoted
		by $\pi_{\lambda}$, of $G= \widetilde{SL}(2,\mathbb{R})$ such that its underlying $(\mathfrak{g}_{\mathbb{C}},K)$-module is isomorphic to $\pi_{K,\lambda}$." which is stated in \cite[Fact 3.27]{MR2956043}  and that ``Any discretely decomposable, infinitesimally unitary $(\mathfrak{g}_{\mathbb{C}},K)$-module is the underlying $(\mathfrak{g}_{\mathbb{C}},K)$-module of a unitary representation of $G$. Furthermore, such a unitary representation is unique." which is stated in \cite[Fact 3.26]{MR2956043} based on T. Kobayashi's theory of discrete decomposable representations \cite{MR1637667,MR1770719}, the claim follows.
	\end{proof}
	
	\vspace{6pt}
	\begin{rmk}[Abuse of notation for $\Omega_b$]
		Let $\rho$ denote the representation of $O(N)$ on
		$L^{2}(\mathbb{R}^{N},|x|^{2b}dx)$ induced by the natural action on $\mathbb{R}^{N}$.
		Since the actions of $\Omega_b$ and $\rho$ commute, they define a representation
		$\Omega_b \boxtimes \rho $
		of $\widetilde{SL}(2,\mathbb{R})\times O(N)$.
		By abuse of notation, we shall also denote this representation by $\Omega_b$ when no confusion is likely to arise.
	\end{rmk}
	
	\begin{cor}[Irreducible decomposition]
		As a representation of $G=\widetilde{SL}(2,\mathbb{R}) \times O(N)$,  $\Omega_{b}$ decomposes as
		\[L^{2}\!\left(\mathbb{R}^{N}, |x|^{2b}dx\right) \cong \sideset{}{^\oplus}{\sum}_{m=0}^{\infty}\pi_{b+\frac{N+2m-2}{2}}\boxtimes \mathcal{H}^{m}(\mathbb{R}^{N}).\]
		
	\end{cor}
	
	\vspace{12pt}
	
	We now explain in what sense the operator $H_b$ and the representation $\Omega_b$ are specified by representation-theoretic data.
	
	\begin{prp}[Characterization of $H_b$ and $\Omega_{b}$]\label{characterization}
		Let $\Omega$ be a unitary representation of 
		$G=\widetilde{SL}(2,\mathbb{R}) \times O(N)$ 
		on $L^2(\mathbb{R}^N, |x|^{2b}dx)$ 
		satisfying the following conditions, and define $H$ by $d\Omega\!\left(\begin{smallmatrix}0&0\\1&0\end{smallmatrix}\right)
		=
		\frac{i}{2}H$. Then $H=H_{b}$ and $\Omega=\Omega_b$.
		
		\begin{enumerate}
			\item
			\[
			d\Omega\!\left(\begin{smallmatrix}1&0\\0&-1\end{smallmatrix}\right)
			=
			E+\frac{N+2b}{2},
			\qquad
			d\Omega\!\left(\begin{smallmatrix}0&1\\0&0\end{smallmatrix}\right)
			=
			\frac{i}{2}|x|^2.
			\]
			\item $O(N)$ acts by the natural action on $\mathbb{R}^N$.
			\item $\Omega$ decomposes as
			\[L^{2}\!\left(\mathbb{R}^{N}, |x|^{2b}dx\right) \cong \sideset{}{^\oplus}{\sum}_{m=0}^{\infty}\pi_{b+\frac{N+2m-2}{2}}\boxtimes \mathcal{H}^{m}(\mathbb{R}^{N}),\]
			where $\pi_{\lambda}$ is the lowest weight representation of lowest weight $\lambda+1$ with respect to the action of $-\frac{H-|x|^{2}}{2}$ and $\mathcal{H}^{m}(\mathbb{R}^{N})$ is the space of spherical harmonics of degree $m$.

		\end{enumerate}
		
	\end{prp}
	
	\begin{proof}
		We set 
		$\mathbf{h}=-\frac{H-|x|^2}{2},\quad \mathbf{e}=\frac{1}{2}\left(E+ \frac{N+2b}{2} - \frac{H+|x|^{2}}{2}\right),\quad \mathbf{f}=\frac{1}{2}\left(E+ \frac{N+2b}{2} + \frac{H+|x|^{2}}{2}\right)$. Then $\mathbf{h},\mathbf{e}, \mathbf{f}$ form an $\mathfrak{sl}_{2}$-triple. That is,
		\begin{align*}
			\left[\mathbf{h},\mathbf{e}\right] &= 2\mathbf{e} \\
			\left[\mathbf{h},\mathbf{f} \right] &= -2\mathbf{f} \\
			\left[\mathbf{e}, \mathbf{f}\right] &= \mathbf{h}. 
		\end{align*}
		(compare with the $\mathfrak{sl}_2$-triple associated to $(\Omega_b,H_b)$ and its action in Remark~\ref{Cayley} and Proposition~\ref{act1}). 
		\vspace{6pt}
		
		Let $F(x)$ be a lowest weight vector of $\pi_{b+\frac{N+2m-2}{2}}\boxtimes \mathcal{H}^{m}(\mathbb{R}^{N})$. Since $\mathbf{f}F(x)=0$ and $\mathbf{h}F(x)=\left(b+\frac{N+2m}{2}\right)F(x)$,
		\[ (\mathbf{h}-2\mathbf{f})F(x)= \left(E+\frac{N+2b}{2}+|x|^{2}\right)F(x) = \left(b+\frac{N+2m}{2}\right)F(x).\] 
		Solving this, we obtain $F(x)\in e^{-\frac{|x|^{2}}{2}}\mathcal{H}^{m}(\mathbb{R}^{N})$. We compute inductively as
		\begin{align*} 
			\MoveEqLeft \mathbf{e}^{l+1}F(x)=\frac{1}{2}\left(E+\frac{N+2b}{2}-|x|^{2}+\mathbf{h}\right)\mathbf{e}^{l}F(x) \\
			& =\left(E-|x|^{2}+2b+N+2l\right)\mathbf{e}^{l}F(x) 
		\end{align*}
		which is independent of the specific form of $H$ and is determined solely by the given conditions. Since $\mathrm{span}_{\mathbb{C}}\left\{ \mathbf{e}^{l}(e^{-\frac{|x|^{2}}{2}}p(x)) \mid l,m\in\mathbb{Z}_{\ge 0},\ p\in \mathcal{H}^{m}(\mathbb{R}^{N}) \right\}$ is dense in $L^{2}\!\left(\mathbb{R}^{N}, |x|^{2b}dx\right)$ and $\mathbf{h}\cdot \mathbf{e}^{l}F(x) = (b+\frac{N+2m}{2}+2l)\mathbf{e}^{l}F(x)$, the action of $H$ is uniquely determined on a dense subspace, hence on the whole space. This proves the proposition.
			\end{proof}
	
	\vspace{12pt}
	
	\subsection{Smooth vectors for $\Omega_b$}
	In this subsection, we determine the space of smooth vectors for $\Omega_b$.
	
	\vspace{6pt}
	
	\begin{prp}[Smooth vectors of $\Omega_{b}$]\label{smooth}
		The space of smooth vectors for $\Omega_b$ is $\mathcal{S}(\mathbb{R}^N)$.
	\end{prp}
	
	\begin{proof}
		We denote the space of smooth vectors for $\Omega_b$ by $W_{b,smooth}$.
		By Proposition~\ref{act1}, we have
		\begin{align*}
			W_{b,smooth}
			=
			\left\{
			\sum_{\ell,j} a_{\ell,j}\Phi_{b,\ell,j}
			\in L^2(\mathbb{R}^N,|x|^{2b}dx)
			:\;
			\begin{aligned}
				&\text{for any $\alpha>0$, there exists $C_\alpha>0$}\\
				&\text{such that } |a_{\ell,j}|
				\le C_\alpha(1+2\ell+m_j)^{-\alpha}
			\end{aligned}
			\right\}.
		\end{align*}
		
		Using the identities for Laguerre polynomials
		$L_\ell^{(\alpha+\beta+1)}(x+y)
		=
		\sum_{k=0}^\ell
		L_{\ell-k}^{(\alpha)}(x)L_k^{(\beta)}(y)$
		and
		$L_\ell^{(\alpha)}(0)
		=
		\frac{\Gamma(\ell+\alpha+1)}{\ell!\Gamma(\alpha+1)}$,
		and rewriting the sum by $\ell=q+r$, we obtain
		\begin{align*}
			\sum_{\ell=0}^\infty a_{\ell,j} L_\ell^{(b+\lambda_{N,m})}\bigl(|x|^2\bigr)
			=
			\sum_{q=0}^\infty
			\left(
			\sum_{r=0}^\infty
			a_{q+r,j}\frac{\Gamma(r+b)}{r!\Gamma(b)}
			\right)
			L_q^{(\lambda_{N,m})}\bigl(|x|^2\bigr).
		\end{align*}
		
		Since $\frac{\Gamma(r+b)}{r!} \thicksim r^{b-1}$ as $r\to\infty$, for any $\alpha>0$ there exists $C'_\alpha>0$ such that
		\[
		\left|
		\sum_{r=0}^\infty
		a_{q+r,j}\frac{\Gamma(r+b)}{r!\Gamma(b)}
		\right|
		\le
		C'_\alpha(1+2q+m_j)^{-\alpha}.
		\]
		
		Thus $W_{b,smooth} \subset W_{0,\mathrm{smooth}}$.
		The reverse inclusion follows by the same argument, and hence
		\[
		W_{b,smooth} = W_{0,\mathrm{smooth}} = \mathcal{S}(\mathbb{R}^N).
		\]
	\end{proof}

	\begin{rmk}\label{conti}
		We define
		\begin{align*}
			\MoveEqLeft H^{s}_{b}(\mathbb{R}^{N}):= \Biggl\{\, \sum_{\ell,j}a_{\ell,j}\Phi_{b,\ell,j} \in L^{2}(\mathbb{R}^{N},|x|^{2b}dx)\hspace{5pt}:\hspace{5pt} \sum_{\ell,j} \left(1+2\ell+m_{j}\right)^{s}\bigr\|a_{\ell,j}\Phi_{b,\ell,j}\bigl\|_{L^2}^{2}  <\infty
			\,\Biggr\} \\
			& \hspace{22pt}= \Biggl\{\, f \in L^{2}(\mathbb{R}^{N},|x|^{2b}dx)\hspace{5pt}:\hspace{5pt} 
			\bigl\| \left(1+|x|^2-H_{b}\right)^{s/2}f \bigr\|_{L^{2}} <\infty
			\,\Biggr\} .
		\end{align*}
		
		Then,
		\begin{align*}
			W_{b,smooth} = \bigcap_{s \in\mathbb{R}} H^{s}_{b}(\mathbb{R}^{N}).
		\end{align*}
		This defines a Fr\'{e}chet topology on $W_{b,smooth}$.
	\end{rmk}
	
	\vspace{12pt}
	
	\subsection{The generalized Fourier transform $\mathcal{F}_b$}\label{2.4}
	In this subsection, we define the generalized Fourier transform $\mathcal{F}_{b}$ and record its basic properties.
	
	\vspace{6pt}
	\begin{dfn}[Generalized Fourier transform]\label{Fo}
		We define the generalized Fourier transform as follows:
		$$\mathcal{F}_{b}:= i^{b+\frac{N}{2}}\,\Omega_{b}(\,e^{\frac{\pi}{2}\bigl(
			\begin{smallmatrix}
				0 & -1 \\
				1 & 0
			\end{smallmatrix}
			\bigl)\,}) = i^{b+\frac{N}{2}}e^{\frac{\pi i}{4}(H_{b}-|x|^2)}. $$
		Here, $i:=e^{\frac{\pi i}{2}}$.
	\end{dfn}
	
	\begin{rmk}[Fourier transform as a Weyl element]\label{Weyl}
		$\mathcal{F}_{b}$ corresponds to a Weyl group element of $\widetilde{SL}_2(\mathbb{R})$. 
	\end{rmk}
	
	\begin{thm}[Properties of the generalized Fourier transform]\label{Fou}
		\begin{enumerate}
			\item \[\mathcal{F}_{b}\Phi_{b,\ell,p}(x) = i^{-(2\ell+m)}\Phi_{b,\ell,p}(x), \]
			where $p(x) \in \mathcal{H}^{m}(\mathbb{R}^{N})$.
			\item The following inversion formula holds:
			\begin{align*}
				\mathcal{F}_{b}^{2}f(x)=f(-x) \hspace{30pt}\text{in particular,}\hspace{30pt}  \mathcal{F}_{b}^{4}=1 
			\end{align*}
			\[ \mathcal{F}_{b}\overline{\mathcal{F}}_{b}=\overline{\mathcal{F}}_{b}\mathcal{F}_{b}=1. \]
		\end{enumerate}
	\end{thm}
	\begin{proof}
		The claim follows from the definition of the generalized Fourier transform $\mathcal{F}_{b}$ and Proposition~\ref{har}.
	\end{proof}
	
	\vspace{12pt}
	
	\subsection{Extension to a holomorphic semigroup}
	In this subsection, we extend $\Omega_{b}$ to a holomorphic semigroup called the Olshanski semigroup. We follow the method of \mbox{Ben Sa\"{i}d}--Kobayashi--{\O}rsted \cite[Section 3.8]{MR2956043} together with \cite[Theorem B]{MR1473443}. 
	\vspace{6pt}

	Let $G=SL(2,\mathbb{R})$, $G_{\mathbb{C}}=SL(2,\mathbb{C})$, and $\widetilde{G}=\widetilde{SL}(2,\mathbb{R})$, and set
	$$
	W := \left\{ \begin{pmatrix}
		a & b \\ c & -a
	\end{pmatrix} :\, a^2+bc\le 0,\; b\ge c \right\}.
	$$
	Then $X\in W$ if and only if $i\,\omega_b(X)\leqq 0$. 
	We define $\Gamma(W):= G\,\mathrm{Exp}(iW)$, which is a subsemigroup of $G_{\mathbb{C}}$. 
	Let $\widetilde{\Gamma}(W)$ be its universal covering. Then
	$\widetilde{\Gamma}(W)= \widetilde{G}\,\mathrm{Exp}(iW) \underset{\mathrm{homeo}}{\cong} \widetilde{G}\times W$.
	The semigroups $\Gamma(W)$ and $\widetilde{\Gamma}(W)$ are called the Olshanski semigroup; they are complex analytic at their interior points.
	
	Since $W= \{0\}\cup \mathrm{Ad}(G)\,\mathbb{R}_{>0}\bigl(\begin{smallmatrix}
		0 & 1 \\ -1 & 0
	\end{smallmatrix}\bigr)\cup \mathrm{Ad}(G)\bigl(\begin{smallmatrix}
		0 & 1 \\ 0 & 0
	\end{smallmatrix}\bigr)$, we have
	$$
	\widetilde{\Gamma}(W) = \widetilde{G}
	\,\cup\,
	\widetilde{G}\,\mathrm{Exp}\!\left(i\,\mathbb{R}_{>0}\bigl(\begin{smallmatrix}
		0 & 1 \\ -1 & 0
	\end{smallmatrix}\bigr)\right)\widetilde{G}
	\,\cup\,
	\widetilde{G}\,\mathrm{Exp}\!\left(i\bigl(\begin{smallmatrix}
		0 & 1 \\ 0 & 0
	\end{smallmatrix}\bigr)\right)\widetilde{G}.
	$$
	
	The interior of $\widetilde{\Gamma}(W)$ coincides with 
	$\widetilde{G}\,\mathrm{Exp}\!\left(i\,\mathbb{R}_{>0}\bigl(\begin{smallmatrix}
		0 & 1 \\ -1 & 0
	\end{smallmatrix}\bigr)\right)\widetilde{G}$,
	which we denote by $\widetilde{\Gamma}(W)_{0}$.

	\begin{prp}[Extension to the holomorphic semigroup]\label{hol}
		The unitary representation $\Omega_{b}: \widetilde{SL}(2,\mathbb{R}) \rightarrow \mathcal{B}\left(L^{2}(\mathbb{R}^{N}, |x|^{2b}\,dx)\right)$ extends to a continuous representation of the Olshanski semigroup $\Omega_{b}: \widetilde{\Gamma}(W) \rightarrow \mathcal{B}\left(L^{2}(\mathbb{R}^{N},|x|^{2b}\,dx)\right)$ which has the following properties: 
		\begin{enumerate}
			\item For any $\gamma\in\widetilde{\Gamma}(W)$, $\|\Omega_{b}(\gamma)\|_{op}\le 1$.
			\item For any $f\in L^2(\mathbb{R}^{N},|x|^{2b}dx)$, the map $\widetilde{\Gamma}(W)_{0}\to \mathbb{C}$, $\gamma \mapsto (\Omega_{b}(\gamma)f,f)$ is holomorphic.
			\item When $\gamma \in \widetilde{\Gamma}(W)_{0}$, $\Omega_{b}(\gamma)$ is a Hilbert--Schmidt operator.
		\end{enumerate}
	\end{prp}
	
	\begin{proof}
		The extension and assertions 1 and 2 follow from  \cite[Theorem B]{MR1473443}.
		
		For 3, by 
		$\widetilde{\Gamma}(W)_{0}= \widetilde{SL}(2,\mathbb{R})\,\mathrm{Exp}\left(\,i\,\mathbb{R}_{> 0}\bigl(\begin{smallmatrix}
			0 & 1 \\
			-1 & 0
		\end{smallmatrix}\bigr)\right)\widetilde{SL}(2,\mathbb{R})$, 
		we need to show that 
		$\Omega_{b}(e^{ti\left(\begin{smallmatrix}
				0 & 1 \\
				-1 & 0
			\end{smallmatrix}\right)})$ is a Hilbert--Schmidt operator for $\mathrm{Re}(t)>0$. 
		This follows from the formula  
		$\Omega_{b}(e^{ti\left(\begin{smallmatrix}
				0 & 1 \\
				-1 & 0
			\end{smallmatrix}\right)}) \Phi_{b,\ell,j}(x) 
		= e^{-t(b+\lambda_{N,m_{j}}+2\ell+1)}\Phi_{b,\ell,j}(x)$.
	\end{proof}
	
	\vspace{12pt}
	
	\subsection{The function $\mathscr{I}_{b,\nu}(w,t)$}\label{2.6}
	
	In Subsection~\ref{2.7}, we describe the action of the representation $\Omega_b$, constructed in Theorem~\ref{lift} and extended in Proposition~\ref{hol}, in terms of integral kernels.
	As a preparation for this, we introduce the auxiliary function $\mathscr{I}_{b,\nu}(w,t)$ and study its basic properties, including an integral representation as in Proposition~\ref{summation} and growth estimates as in Propositions~\ref{bound} and \ref{bound2}.
	
	\vspace{12pt}
	
	The kernel formulas in the next subsection are naturally expressed in terms of the following function.
	\begin{dfn}[The function $\mathscr{I}_{b,\nu}(w,t)$]\label{Ibt}
		Assume \(\nu>-1\) and \(b>-\nu-1\). For $w\in\mathbb C$ and $t\in[-1,1]$, we define
		\[
		\mathscr{I}_{b,\nu}(w,t)
		:=
		\sum_{m=0}^{\infty}
		\frac{\Gamma(b+\nu+1)}{\Gamma(b+\nu+m+1)}
		\left(\frac{w}{2}\right)^m
		\mathcal{I}_{b+\nu+m}(w)\,
		\widetilde{C}^{(\nu)}_{m}(t).
		\]
	\end{dfn}
	
	Here
	\[
	\mathcal{I}_{\nu}(w)
	:=
	\sum_{m=0}^{\infty}
	\frac{(w/2)^{2m}}{(\nu+1)_m\,m!}
	\]
	denotes the normalized modified Bessel function, and
	\[
	\widetilde{C}_{m}^{(\nu)}(t)
	:=
	\frac{m+\nu}{\nu}
	\sum_{k=0}^{\lfloor m/2\rfloor}
	(-1)^k
	\frac{\Gamma(m-k+\nu)}{\Gamma(\nu)\,k!\,(m-2k)!}
	(2t)^{m-2k}
	\]
	denotes the normalized Gegenbauer polynomial.
	
	\begin{rmk}
		When $b=0$, the function $\mathscr{I}_{b,\nu}$ reduces to
		\[
		\mathscr{I}_{0,\nu}(w,t)=e^{wt},
		\]
		which follows from the Gegenbauer expansion of the exponential function; see, e.g., \cite[7.13(14)]{MR1034956}. See also \cite[(4.40) and (4.45)]{MR2956043}.
	\end{rmk}
	
	\begin{prp}[Explicit formula for $\mathscr{I}_{b,\nu}(w,t)$]\label{summation}
		For $b>0, \nu>-1, w\in\mathbb C$, and $t\in[-1,1]$, 
		\[
		\mathscr{I}_{b,\nu}(w,t)
		=
		\frac{1}{B(b,\nu+1)}
		\int_0^1
		u^{b-1}(1-u)^\nu
		\mathcal{I}_b\bigl(uw\bigr)\,e^{(1-u)wt}\,du.
		\]
	\end{prp}
	
	\vspace{12pt}
	
	\begin{proof}

	We begin with a lemma on Bessel functions.
	\begin{lmm}\label{bes}
		For $\mathrm{Re}(\nu)>0$, $\mathrm{Re}(\mu)>-1$, and $w>0$,
		\[
		\int_{0}^{w} J_{\nu}(x)J_{\mu}(w-x)\,\frac{dx}{x}
		=
		\frac{1}{\nu}J_{\nu+\mu}(w).
		\]
	\end{lmm}
	
	\begin{proof}
		Put $f_\nu(x)=J_\nu(x)/x$ and $g_\mu(x)=J_\mu(x)$.
		Then the left-hand side is $(f_\nu*g_\mu)(w)$.
		Taking the Laplace transform,
		$\mathcal L(f_\nu*g_\mu)
		=
		\mathcal L(f_\nu)\mathcal L(g_\mu)$.
		Using $\mathcal L(J_\alpha)(s)
		=
		\frac{(\sqrt{1+s^2}-s)^\alpha}{\sqrt{1+s^2}}$, and $\mathcal L(J_\alpha/x)(s)
		=
		\frac{1}{\alpha}(\sqrt{1+s^2}-s)^\alpha$ ,
		we obtain
		$\mathcal L(f_\nu*g_\mu)
		=
		\frac{1}{\nu}
		\frac{(\sqrt{1+s^2}-s)^{\nu+\mu}}{\sqrt{1+s^2}}
		=
		\mathcal L\!\left(\frac{1}{\nu}J_{\nu+\mu}\right)$.
		The result follows from uniqueness of the Laplace transform.
	\end{proof}
	
	\vspace{12pt}
		
		We now return to the proof.
		
	\vspace{6pt}
	
		By Lemma~\ref{bes},
		\[
		\mathcal{I}_{\alpha+\beta}(w)
		=
		\frac{1}{B(\alpha,\beta+1)}
		\int_0^1
		u^{\alpha-1}(1-u)^\beta
		\mathcal{I}_\alpha\bigl(uw\bigr)\mathcal{I}_\beta\bigl((1-u)w\bigr)\,du.
		\]
		
		Hence,
		
		\begin{align*}
			\MoveEqLeft \mathscr{I}_{b,\nu}(w,t) \\
			&=
			\sum_{m=0}^{\infty}
			\frac{\Gamma(b+\nu+1)}{\Gamma(b+\nu+m+1)}
			\left(\frac{w}{2}\right)^m
			\mathcal{I}_{b+\nu+m}(w)\,
			\widetilde{C}_m^{(\nu)}(t) \\
			&=
			\frac{1}{B(b,\nu+1)}
			\sum_{m=0}^{\infty}
			\frac{\Gamma(\nu+1)}{\Gamma(\nu+m+1)}
			\left(\frac{w}{2}\right)^m
			\widetilde{C}_m^{(\nu)}(t)
			\int_0^1
			u^{b-1}(1-u)^{\nu+m}
			\mathcal{I}_b\bigl(uw\bigr)\mathcal{I}_{\nu+m}\bigl((1-u)w\bigr)\,du \\
			&=
			\frac{1}{B(b,\nu+1)}
			\int_0^1
			u^{b-1}(1-u)^\nu
			\mathcal{I}_b\bigl(uw\bigr)\,
			\mathscr{I}_{0,\nu}\bigl((1-u)w,t\bigr)\,du \\
			&=
			\frac{1}{B(b,\nu+1)}
			\int_0^1
			u^{b-1}(1-u)^\nu
			\mathcal{I}_b\bigl(uw\bigr)\,e^{(1-u)wt}\,du.
		\end{align*}
	\end{proof}
	
	\vspace{12pt}
	
	We next record estimates for $\mathscr I_{b,\nu}(w,t)$ that will be needed in Subsection~\ref{2.7}.
	\begin{prp}[A bound for $\mathscr{I}_{b,\nu}(w,t)$]\label{bound}
		Assume \(\nu>-1\), \(t\in[-1,1]\), \(w\in\mathbb C\) and \(b\ge 0\). Then,  
		\[
		|\mathscr{I}_{b,\nu}(w,t)|\le e^{|\mathrm{Re}(w)|}.
		\]
	\end{prp}
	
	\vspace{12pt}
	
	\begin{proof}
		By Proposition~\ref{summation}, for $b>0$ we have
		\begin{align*}
			\bigl|\mathscr{I}_{b,\nu}(w,t)\bigr|
			&\le
			\frac{1}{B(b,\nu+1)}
			\int_0^1
			u^{b-1}(1-u)^\nu
			\bigl|\mathcal{I}_b\bigl(uw\bigr)\bigr|
			\bigl|e^{(1-u)wt}\bigr|\,du \\
			&\le
			\frac{1}{B(b,\nu+1)}
			\int_0^1
			u^{b-1}(1-u)^\nu
			e^{|\mathrm{Re}(w)|u}e^{|\mathrm{Re}(w)|(1-u)}\,du \\
			&=
			e^{|\mathrm{Re}(w)|}.
		\end{align*}
		Here, we used the inequality $ |\mathcal{I}_b(z)|\le e^{|\mathrm{Re}(z)|}\, \left(b>-\frac12\right),$
		which follows from the integral representation
		$\mathcal{I}_b(z)
		=
		\frac{1}{B\!\left(b+\frac12,\frac12\right)}
		\int_{-1}^1 e^{zs}(1-s^2)^{b-\frac12}\,ds$.
	\end{proof}
	
	\vspace{12pt}
	
	The preceding estimate extends to the range $b>-\nu-1$ after a slight modification of the argument.
	
	\begin{prp}[Polynomial-exponential bound for \(\mathscr I_{b,\nu}(w,t)\)]\label{bound2}
		Assume \(\nu>-1\), \(t\in[-1,1]\), \(w\in\mathbb C\) and \(b>-\nu-1\). Then there exist
		constants \(C_{b,\nu}>0\) and \(M_{b,\nu}\ge0\) such that
		\[
		\bigl|\mathscr I_{b,\nu}(w,t)\bigr|
		\le
		C_{b,\nu}(1+|w|)^{M_{b,\nu}}e^{|\mathrm{Re}\,w|}.
		\]
	\end{prp}
	
	\begin{proof}
		The proof is somewhat technical and is therefore deferred to
		Appendix~\ref{4.1}. 
		
		The argument starts from the integral formula for
		$\mathscr I_{b,\nu}(w,t)$ when $b>0$, and expands the integrand at $u=0$
		into a finite Taylor part and a remainder. This yields a decomposition
		formula for $\mathscr I_{b,\nu}(w,t)$ into finitely many explicit terms
		and a remainder term. The resulting identity is then extended to the full
		range $b>-\nu-1$ by analytic continuation in $b$, and the desired estimate
		follows by bounding the explicit terms and the remainder separately.
	\end{proof}
	\vspace{12pt}

	\begin{rmk}
		We record several explicit formulas for $\mathscr{I}_{b,\nu}(w,t)$ without proof.
		
		For $b>0, \nu>-1, w\in\mathbb C$, and $t\in[-1,1]$,
		\begin{align*}
			\MoveEqLeft \mathscr{I}_{b,\nu}(w,t) \\
			&=
			\sum_{m=0}^{\infty}
			\frac{\Gamma(b+\nu+1)}{\Gamma(b+\nu+m+1)}
			\left(\frac{w}{2}\right)^m
			\mathcal{I}_{b+\nu+m}(w)\,
			\widetilde{C}^{(\nu)}_{m}(t) \\
			&=
			\frac{1}{B(b,\nu+1)}
			\int_0^1
			u^{b-1}(1-u)^\nu
			\mathcal{I}_b(uw)\,e^{(1-u)wt}\,du \\
			&=
			\sum_{m,n=0}^{\infty}
			\frac{(\nu+1)_{n}(b)_{2m}}{(b+\nu+1)_{2m+n}(b+1)_{m}}
			\frac{(w/2)^{2m}(wt)^{n}}{m!\,n!} \\
			&=
			\frac{1}{B(b,\nu+1)B(b+\frac{1}{2},\frac{1}{2})}
			\int_{(u,s)\in [0,1]\times [-1,1]}
			(1-u)^{\nu}u^{b-1}(1-s^{2})^{b-\frac{1}{2}}
			e^{(1-u)wt+suw}\,ds\,du \\
			&=
			e^{wt}
			F_{0,1,1}\!\left(
			\begin{smallmatrix}
				b+\frac{1}{2},\,b\\
				2b+1,\,b+\nu+1
			\end{smallmatrix}
			;\,
			\begin{smallmatrix}
				w-wt\\
				-2w
			\end{smallmatrix}
			\right),
		\end{align*}
		
		where
		$F_{0,1,1}\!\left(
		\begin{smallmatrix}
			\alpha,\,\beta\\
			\gamma,\,\delta
		\end{smallmatrix}
		;\,
		\begin{smallmatrix}
			X\\
			Y
		\end{smallmatrix}
		\right)
		:=
		\sum_{m,n=0}^{\infty}
		\frac{(\alpha)_n(\beta)_{m+n}}{(\gamma)_n(\delta)_{m+n}}
		\frac{X^m Y^n}{m!\,n!}$.
	\end{rmk}
	
	\vspace{12pt}
	
	\subsection{Integral kernels for $\Omega_b$}\label{2.7}
	
	In this subsection, we describe the action of $\Omega_b$ by integral kernels; see Theorem~\ref{ker}. 
	
	We first introduce the kernel $\Lambda_b(x,y;t)$ associated with the semigroup generated by $\frac12(H_b-|x|^2)$, which might be viewed as a generalized Mehler-type kernel. From this, we obtain in particular an explicit integral kernel $B_{b}(x,y)$ for the generalized Fourier transform $\mathcal{F}_b$. 
	
	We also define the kernel $h_b(x,y;t)$ of the semigroup $e^{\frac t2 H_b}$, which might be regarded as a heat-kernel-type expression associated with $H_b$. These kernels give explicit realizations of the action of several distinguished elements of the representation $\Omega_b$; see Remark~\ref{arbitrary}.
	
	\vspace{18pt}
	
	Let $\mathbb{C}^+ := \{ z \in \mathbb{C} \mid \mathrm{Re} (z) \ge 0 \}$ and $c_{b,N}:=\frac{1}{(2\pi)^{b+N/2}} \frac{\Gamma(N/2)\pi^{b}}{\Gamma(b+N/2)}$.\\
	By Schwartz's kernel theorem, there exists a distribution kernel $\Lambda_{b}(x,y;t)$ satisfying the following:\\
	 For $t \in \mathbb{C}^{+}$, and $f(x)\in\mathcal{S}(\mathbb{R}^{N})$,
	 \[\Omega_{b}(e^{ti\left(\begin{smallmatrix}
			0 & 1 \\
			-1 & 0
		\end{smallmatrix}\right)})f(x)= e^{\frac{t}{2}(H_{b}-|x|^2)}f(x)= c_{b,N}\int_{\mathbb{R}^{N}}\Lambda_{b}(x,y;t)f(y) \,|y|^{2b}dy.\] 
		We denote $B_{b}(x,y):=i^{b+N/2}\Lambda_{b}(x,y;\frac{\pi i}{2})$. Then, for $f(x)\in\mathcal{S}(\mathbb{R}^{N})$,
		\[\mathcal{F}_{b}f(x) = c_{b,N}\int_{\mathbb{R}^{N}} B_{b}(x,y)f(y) \,|y|^{2b}dy .\] 
		We also define the distribution kernel $h_{b}(x,y; t)$ satisfying the following: \\
		For $t \in \mathbb{C}^{+}$, and $f(x)\in\mathcal{S}(\mathbb{R}^{N})$,
		 \[\Omega_{b}(e^{ti\left(\begin{smallmatrix}
			0 & 0 \\
			-1 & 0
		\end{smallmatrix}\right)})f(x)=e^{\frac{t}{2}H_{b}}f(x)=c_{b,N}\int_{\mathbb{R}^{N}}h_{b}(x,y; t)f(y)\,|y|^{2b}dy .\] 
	
		\vspace{18pt}
		
		The following theorem gives explicit formulas for these kernels in terms of the function $\mathscr I_{b,\nu}$ introduced in Subsection~\ref{2.6}.
		
		\begin{thm}[Explicit formulas for the integral kernels]\label{ker}
			The kernels are given by
			\begin{align*}
				\Lambda_b(x,y;t)
				&=
				\frac{1}{\sinh(t)^{b+N/2}}
				e^{-\frac{|x|^2+|y|^2}{2\tanh(t)}}
				\mathscr I_{b,\frac{N-2}{2}}\!\left(\frac{|x||y|}{\sinh(t)},\frac{\langle x,y\rangle}{|x||y|}\right),
				&& t\in \mathbb C^+\setminus 2\pi i\mathbb Z,\\[6pt]
				h_b(x,y;t)
				&=
				\frac{1}{t^{b+N/2}}
				e^{-\frac{|x|^2+|y|^2}{2t}}
				\mathscr I_{b,\frac{N-2}{2}}\!\left(\frac{|x||y|}{t},\frac{\langle x,y\rangle}{|x||y|}\right),
				&& t\in \mathbb C^+,\\[6pt]
				B_b(x,y)
				&=
				\mathscr I_{b,\frac{N-2}{2}}\!\left(-i|x||y|,\frac{\langle x,y\rangle}{|x||y|}\right).
			\end{align*}
			Here $\mathscr{I}_{b,\nu}$ is as in Definition~\ref{Ibt}; see Proposition~\ref{summation} for its explicit formula.
		\end{thm}
	
	\vspace{6pt}
	\begin{proof}
		\begin{enumerate}
			\item (Computation of $\Lambda_b(x,y;t)$) 
			
			Fix $t \in \mathbb{C}^{+}\setminus i\mathbb{R}$. Since 
			$\Omega_{b}(e^{ti\left(\begin{smallmatrix}
					0 & 1 \\
					-1 & 0
				\end{smallmatrix}\right)})$ is a Hilbert-Schmidt operator by Proposition~\ref{hol}, Proposition~\ref{har}  gives the expansion
			\[c_{b,N}\Lambda_{b}(x,y;t) = \sum_{\ell,j=0}^{\infty} e^{-t(b+\lambda_{N,m_{j}}+2\ell+1)}  \frac{\Phi_{b,\ell,j}(x)\overline{\Phi_{b,\ell,j}(y)}}{\left\|\Phi_{b,\ell,j}\right\|^{2}_{L^{2}}}\] 
			with convergence in $L^2(\mathbb{R}^{N}\times \mathbb{R}^{N}, |x|^{2b}|y|^{2b}dxdy)$. 
			
			By the Hille-Hardy formula 
			\[ \sum_{\ell=0}^{\infty} \frac{\Gamma(\nu+1)\Gamma(\ell+1)}{\Gamma(\nu+\ell+1)}L^{(\nu)}_{\ell}(X)L^{(\nu)}_{\ell}(Y)T^{\ell} = (1-T)^{-(\nu+1)}e^{-\frac{(X+Y)T}{1-T}}\mathcal{I}_{\nu}\left(\frac{2\sqrt{XYT}}{1-T}\right)\] and the addition formula for the zonal spherical harmonics \[\sum_{m_{j}=m}p_{j}(\omega)\overline{p_{j}(\mu)} = \frac{1}{\mathrm{vol}(S^{N-1})}\widetilde{C}_{m}^{(\frac{N-2}{2})}(\langle\omega,\mu\rangle),\]
			
			one obtains
			
			\[\Lambda_{b}(x,y; t) =  \frac{1}{\left(\sinh(t)\right)^{b+N/2}}e^{-\frac{|x|^2+|y|^2}{2\tanh(t)}}\mathscr{I}_{b,\frac{N-2}{2}}\left(\frac{|x||y|}{\sinh(t)},\frac{\langle x,y\rangle}{|x||y|}\right).\]
			
			Here 
			$\mathcal{I}_{\nu}(w)$
			denotes a normalized modified Bessel function, and
			$\widetilde{C}_{m}^{(\nu)}(t)$
			denotes the normalized Gegenbauer polynomial; see Definition~\ref{Ibt}.
			
			\vspace{12pt}
			Suppose $it\in i\mathbb{R}\setminus 2\pi i\mathbb{Z}$. 
			Since the representation $\Omega_{b}$ is continuous as in Proposition~\ref{hol}, 
			for any $f\in L^{2}(\mathbb{R}^{N},|x|^{2b}dx),\quad\Omega_{b}(e^{t \left(\begin{smallmatrix}
					0 & -1 \\
					1 & 0
				\end{smallmatrix}\right)})f = \lim_{\varepsilon\to + 0} \Omega_{b}(e^{(t+i\varepsilon) \left(\begin{smallmatrix}
					0 & -1 \\
					1 & 0
				\end{smallmatrix}\right)})f$.

			By Proposition~\ref{bound2}, there exists a constant $C>0$ such that
			\begin{align*}
				\left|\Lambda_{b}(x,y ; t)\right| \,\le \,C\,\frac{1}{\left|\sinh(t)\right|^{b+N/2}}\left(1+\left(\frac{|x||y|}{\sinh(t)}\right)^{N/2}\right).
			\end{align*}
			Thus, applying Lebesgue's dominated convergence theorem to $\Omega_{b}(e^{t \left(\begin{smallmatrix}
					0 & -1 \\
					1 & 0
				\end{smallmatrix}\right)})f $ when $f\in \mathcal{S}(\mathbb{R}^{N})$, we obtain $\lim_{\varepsilon\to 0} \Lambda_{b}(x,y; it+\varepsilon)=\Lambda_{b}(x,y; it)$.
			
			\item (Computation of $B_b(x,y)$) 
				This follows immediately from the definition 
				\[B_b(x,y):=i^{b+N/2}\Lambda_b\left(x,y;\frac{\pi i}{2}\right)\] 
				and the formula for $\Lambda_b(x,y;t)$ proved in 1.
			\item (Computation of $h_b(x,y;t)$) 
			
			Fix $t\in\mathbb{C}^{+}$. $\begin{pmatrix}
				1 & 0\\t&1 \\
			\end{pmatrix} = \lim_{\varepsilon\to + 0} \begin{pmatrix} \varepsilon^{1/2}&0\\0&\varepsilon^{-1/2}\end{pmatrix}\begin{pmatrix}\cosh(\varepsilon t)& \sinh(\varepsilon t)\\\sinh(\varepsilon t) & \cosh(\varepsilon t)\end{pmatrix}\begin{pmatrix}\varepsilon^{-1/2}& 0\\0&\varepsilon^{1/2}\end{pmatrix}$. 
			
			Since $\Omega_{b}( \begin{pmatrix} \varepsilon&0\\0&\varepsilon^{-1}\end{pmatrix})f(x)=\varepsilon^{b+\frac{N}{2}}f(\varepsilon x)$, and continuity of the representation as in Proposition~\ref{hol},
			\begin{align*}
				\MoveEqLeft\Omega_{b}(\begin{pmatrix}
					1&0\\t&1
				\end{pmatrix})f(x) = \lim_{\varepsilon\to + 0} c_{b,N}\int_{\mathbb{R}^{N}} \Lambda_{b}(\varepsilon^{1/2}x,y;\varepsilon t)f(\varepsilon^{-1/2} y) |y|^{2b}dy &\\
				& = \lim_{\varepsilon\to + 0} c_{b,N}\int_{\mathbb{R}^{N}} \varepsilon^{b+N/2}\Lambda_{b}(\varepsilon^{1/2}x,\varepsilon^{1/2}y; \varepsilon t)f(y) |y|^{2b}dy. 
			\end{align*} 
			By Proposition~\ref{bound2}, there exists a constant $C>0$ such that
			\begin{align*}
				\left|\varepsilon^{b+N/2}\Lambda_{b}(\varepsilon^{1/2}x,\varepsilon^{1/2}y; \varepsilon t)\right| \le C\frac{\varepsilon^{b+N/2}}{\left|\sinh(\varepsilon t)\right|^{b+N/2}}\left(1+\left(\frac{\varepsilon|x||y|}{\sinh(\varepsilon t)}\right)^{N/2}\right).
			\end{align*}
			Applying Lebesgue's dominated convergence theorem when $f(x)\in \mathcal{S}(\mathbb{R}^{N})$, we obtain the claim. 
			
		\end{enumerate}
		
	\end{proof}
	
	\vspace{12pt}
	
	\begin{rmk}[The case of $N=1$]\label{N1}
		We consider the case when $N=1$ of Theorem~\ref{ker}. Although only the terms $m=0,1$ contribute, the Theorem remains valid:
		
		If $m\ge2$,  $C_{m}^{(-\frac{1}{2})}(\pm1)=0$. This follows from the generating function formula of the Gegenbauer polynomials $(1-2tx+x^2)^{-\nu} = \sum_{m=0}^{\infty}C_{m}^{(\nu)}(t)x^{m}$. 
		By this, when $N=1$, 
		\begin{align*}
			\MoveEqLeft \mathscr{I}_{b,\frac{N-2}{2}}\left(\frac{|x||y|}{\mathrm{sinh}(t)},\frac{\langle x,y\rangle}{|x||y|}\right) &\\
			& =\sum_{m=0}^{\infty}\frac{\Gamma(b+\frac{N}{2})}{\Gamma(b+\frac{N}{2}+m)}\left(\frac{|x||y|}{2\mathrm{sinh}(t)}\right)^{m}\mathcal{I}_{b+\lambda_{N,m}}\left(\frac{|x||y|}{\mathrm{sinh}(t)}\right)\widetilde{C}^{(\frac{N-2}{2})}_{m}\left(\frac{\langle x,y\rangle}{|x||y|}\right)\\
			& =\sum_{m=0}^{1}\frac{\Gamma(b+\frac{1}{2})}{\Gamma(b+\frac{1}{2}+m)}\left(\frac{|x||y|}{2\mathrm{sinh}(t)}\right)^{m}\mathcal{I}_{b-\frac{1}{2}+m}\left(\frac{|x||y|}{\mathrm{sinh}(t)}\right)\widetilde{C}^{(-\frac{1}{2})}_{m}\left(\frac{\langle x,y\rangle}{|x||y|}\right).
		\end{align*}
		
		Thus the calculation in Theorem~\ref{ker} is correct when $N=1$. 
		
		Each integral kernel is given by
		\begin{align*}
			\Lambda_{b}(x,y;t) = \Gamma(b+1/2)\,\mathrm{sinh}(t)^{-\left(\frac{1}{2}+b\right)} e^{-\frac{x^2+y^2}{2\mathrm{tanh}(t)}}\left(\widetilde{I}_{b-\frac{1}{2}}\left(\frac{xy}{\mathrm{sinh}(t)}\right) + \frac{xy}{2\mathrm{sinh}(t)} \,\widetilde{I}_{b+\frac{1}{2}}\left(\frac{xy}{\mathrm{sinh}(t)}\right)\right) 
		\end{align*}
		\begin{align*}
			\MoveEqLeft h_{b}(x,y;t) = \Gamma(b+1/2)\,t^{-\left(\frac{1}{2}+b\right)} e^{-\frac{x^2+y^2}{2t}}\left(\widetilde{I}_{b-\frac{1}{2}}\left(\frac{xy}{t}\right) + \frac{xy}{2t} \,\widetilde{I}_{b+\frac{1}{2}}\left(\frac{xy}{t}\right)\right) 
		\end{align*}
		\begin{align*}
			\MoveEqLeft  B_{b}(x,y) = \Gamma(b+1/2)\, \left(\widetilde{J}_{b-\frac{1}{2}}\left(xy\right) -i\frac{xy}{2} \,\widetilde{J}_{b+\frac{1}{2}}\left(xy\right)\right),
		\end{align*}
		
		where $\widetilde{I}_{\nu}(x):= \sum_{m=0}^{\infty}\frac{(x/2)^{2m}}{\Gamma(\nu+m+1)m!}$, $\widetilde{J}_{\nu}(x):= \sum_{m=0}^{\infty}\frac{(-1)^{m}(x/2)^{2m}}{\Gamma(\nu+m+1)m!}$ denote normalized Bessel functions.
		
	\end{rmk}
	
	\vspace{12pt}

	\begin{rmk}[Action of $\Omega_b$ for arbitrary elements of $\widetilde{\Gamma}(W)$]\label{arbitrary}
		The action of $\widetilde{\Gamma}(W)$ can be computed explicitly
		from Theorem~\ref{ker} by elementary computations. 
		
		Let $\widetilde{G}=\widetilde{SL}(2,\mathbb{R})$ and $\widetilde{K}=\widetilde{SO}(2)$.
		Set
		\[
		\begin{aligned}
			&M=\left\{\exp\!\left(m\pi\begin{psmallmatrix}0&-1\\1&0\end{psmallmatrix}\right)\right\},\quad
			A=\left\{\exp\!\left(\begin{psmallmatrix}a&0\\0&-a\end{psmallmatrix}\right)\right\},\quad
			N=\left\{\exp\!\left(\begin{psmallmatrix}0&b\\0&0\end{psmallmatrix}\right)\right\},\\
			&\widetilde{K}_{\mathbb{C}^+}
			=\left\{\exp\!\left(it\begin{psmallmatrix}0&1\\-1&0\end{psmallmatrix}\right)\right\},\quad
			N_{\mathbb{C}^+}
			=\left\{\exp\!\left(\begin{psmallmatrix}0&ib\\0&0\end{psmallmatrix}\right)\right\},\quad
			\overline{N_{\mathbb{C}^+}}
			=\left\{\exp\!\left(\begin{psmallmatrix}0&0\\-ib&0\end{psmallmatrix}\right)\right\},
		\end{aligned}
		\]
		and put $P=MAN$.
		
		Recall that
		$\widetilde{\Gamma}(W) =
		\widetilde{G}
		\cup
		\widetilde{G}\,\mathrm{Exp}\!\left(i\,\mathbb{R}_{>0}
		\begin{psmallmatrix}0&1\\-1&0\end{psmallmatrix}\right)\widetilde{G}
		\cup
		\widetilde{G}\,\mathrm{Exp}\!\left(i
		\begin{psmallmatrix}0&1\\0&0\end{psmallmatrix}\right)\widetilde{G}$.
		Using the Iwasawa decomposition $\widetilde{G}=\widetilde{K}AN$
		and the Bruhat decomposition $\widetilde{G}=P\cup Pw_0P$,
		we obtain the decomposition :
		\[
		\widetilde{\Gamma}(W)
		=
		\widetilde{K}AN
		\cup
		(AN)\,\widetilde{K}_{\mathbb{C}^+}\,(AN)
		\cup
		\bigl(
		P N_{\mathbb{C}^+} P
		\cup
		P w_0 N_{\mathbb{C}^+} P
		\cup
		P N_{\mathbb{C}^+} w_0 P
		\cup
		P \overline{N_{\mathbb{C}^+}} P
		\bigr),
		\]
		where $w_0=\exp\!\left(\frac{\pi}{2}\begin{psmallmatrix}0&-1\\1&0\end{psmallmatrix}\right)$.
		
		By this, the actions of $\Omega_{b}$ reduce to that of
		$\widetilde{K}_{\mathbb{C}^+}$, $\overline{N_{\mathbb{C}^+}}$, and $w_0$,
		which are described by Theorem~\ref{ker},
		up to the elementary actions of $M$, $A$, and $N_{\mathbb{C}^+}$:
		\begin{align*}
			\Omega_b\!\left(e^{m\pi\begin{psmallmatrix}0&-1\\1&0\end{psmallmatrix}}\right)f(x)
			&= e^{-m\pi i(b+N/2)} f((-1)^m x),\\
			\Omega_b\!\left(e^{\begin{psmallmatrix}a&0\\0&-a\end{psmallmatrix}}\right)f(x)
			&= e^{a(b+N/2)} f(e^a x),\\
			\Omega_b\!\left(e^{\begin{psmallmatrix}0&ib\\0&0\end{psmallmatrix}}\right)f(x)
			&= e^{-\frac{b|x|^2}{2}} f(x).
		\end{align*}
	\end{rmk}
	\vspace{12pt}

	\begin{cor}[Bound for the integral kernels]
		Suppose $b \ge0$. 
		\begin{enumerate}
			\item 
			$$ \left|B_{b}(x,y)\right| \hspace{2pt}\le\hspace{2pt} 1.$$
			
			\item For $t\hspace{1pt}>\hspace{1pt}0$,
			$$ 0\hspace{3pt}\le\hspace{3pt} h_{b}(x,y;t) \hspace{3pt}\le\hspace{3pt} t^{-(b+N/2)}e^{-\frac{||x|-|y||^{2}}{2t}}.$$
			
		\end{enumerate}
	\end{cor}
	\begin{proof}
		
		This follows from Proposition~\ref{bound} and Theorem~\ref{ker}.
		
			\end{proof}
	
	\vspace{12pt}
	
	\section{Generalized derivatives $D_{b,n}$}\label{3}
	In this section, we introduce operators $D_{b,n}$ as deformations of the partial derivatives, arising from the representation-theoretic framework constructed in Section~\ref{2}. 
	
	We study their basic properties in Subsections~\ref{3.1} and \ref{3.2}, and derive their explicit formula in Subsection~\ref{3.3}. This explicit formula might be viewed as an analogue of Dunkl operators \cite{MR951883}, formally corresponding to the case where the reflection group is $O(N)$ and the root system is $\Phi=S^{N-1}$.
	
	\subsection{Definition of $D_{b,n}$}\label{3.1}
	In this subsection, we define the operators $D_{b,n}$.
	
	\begin{dfn}[The operators $D_{b,n}$]\label{Db}
		We define the operators $D_{b,n}$ ($n=1,\dots, N$) on $L^{2}(\mathbb{R}^{N},|x|^{2b}dx)$ with domain $W_{b,smooth}=\mathcal{S}(\mathbb{R}^{N})$ by
		\[ D_{b,n}:=\left[\frac{H_{b}-|x|^{2}}{2}, x_{n}\right].\]
		Here $H_{b}$ is the deformation of the Laplacian defined in Definition~\ref{Hb}.
	\end{dfn}
	This definition is motivated by the classical relation between the Laplacian and the partial derivatives. We will derive an explicit formula for $D_{b,n}$ in Subsection~\ref{3.3}. By definition, $D_{b,n}W_{b,alg}\subset W_{b,alg}$ and $D_{b,n}W_{b,smooth}\subset W_{b,smooth}$.
	
	\vspace{20pt}
	The next proposition describes the action of $x_n$ and $D_{b,n}$ on the basis vectors in $W_{b,alg}$, and will be used in Subsection~\ref{3.2} to prove the Fourier-intertwining relations.
	\begin{prp}[Action of $D_{b,n}$ and $x_{n}$ on $W_{b,alg}$]\label{act2}
		\begin{align*}
			\MoveEqLeft\frac{x_{n} + D_{b,n}}{2}\Phi_{b,\ell,p} 
			= -\Phi_{b\,,\,\ell-1,\,x_{n}p-\frac{1}{2\lambda_{N,m}}|x|^2\frac{\partial p}{\partial x_{n}} } 
			+ (b+\lambda_{N,m}+\ell) \Phi_{b,\,\ell,\,\frac{1}{2\lambda_{N,m}}\frac{\partial p}{\partial x_{n}} } \\
			\MoveEqLeft\frac{x_{n} - D_{b,n}}{2} \Phi_{b,\ell,p} = \Phi_{b\,,\,\ell,\,x_{n}p-\frac{1}{2\lambda_{N,m}}|x|^2\frac{\partial p}{\partial x_{n}} } 
			- (\ell+1) \Phi_{b,\,\ell+1,\,\frac{1}{2\lambda_{N,m}}\frac{\partial p}{\partial x_{n}} } 
		\end{align*}
	\end{prp}
	\begin{proof}
		Let $p(x) \in \mathcal{H}^m(\mathbb{R}^N)$. We decompose it as
		\begin{align*}
			\MoveEqLeft x_{n}\,p(x) = \left(x_{n}\,p(x)- \frac{1}{2\lambda_{N,m}}|x|^2\frac{\partial p}{\partial x_{n}}(x)\right) + \frac{1}{2\lambda_{N,m}}|x|^2\frac{\partial p}{\partial x_{n}}(x).
		\end{align*}
		Then, $x_{n}p(x)- \frac{1}{2\lambda_{N,m}}|x|^2\frac{\partial p}{\partial x_{n}}(x) $ is a $(m+1)$-th harmonic polynomial and
		$\frac{\partial p}{\partial x_{n}}(x)$ is a $(m-1)$-th harmonic polynomial. Using this and the following identities for Laguerre polynomials
		\begin{align*}
			\MoveEqLeft L^{(\alpha+1)}_{\ell}(t)-L^{(\alpha+1)}_{\ell-1}(t)=L^{(\alpha)}_{\ell}(t) \\
			\MoveEqLeft tL_{\ell}^{(\alpha+1)}(t) = -(\ell+1)L_{\ell +1}^{(\alpha)}(t)+(\ell+\alpha +1)L^{(\alpha)}_{\ell}(t),
		\end{align*}
		the claim follows.
	\end{proof}
	\vspace{12pt}
	
	\subsection{Interchange of $D_{b,n}$ and $x_n$ under $\mathcal{F}_b$}\label{3.2}
	
	In this subsection, we show that the generalized Fourier transform $\mathcal{F}_b$ interchanges $x_n$ and $D_{b,n}$ in Theorem~\ref{fou1}. This leads to the commutativity of the operators $D_{b,n}$, the quadratic relations in Corollary~\ref{relation}, and the standard action of $\widetilde{SL}(2,\mathbb{R})$ on $V_{b,n}:=\{x_n,iD_{b,n}\}_{\mathbb R}$ as in Corollary~\ref{lift2}.
	
	\vspace{6pt}
	
	\begin{thm}[Interchange of $D_{b,n}$ and $x_n$ under $\mathcal{F}_b$]\label{fou1}
		Let $\mathcal{F}_b$ be the generalized Fourier transform defined in Definition~\ref{Fo}, and let $D_{b,n}$ be the operators defined in Definition~\ref{Db}. Then
		\[\mathcal{F}_{b}D_{b,n}\mathcal{F}_{b}^{-1}	=  ix_{n}, \hspace{36pt}\mathcal{F}_{b}x_{n}\mathcal{F}_{b}^{-1}	=  iD_{b,n}.\]
	\end{thm}
	\begin{proof}
		By Proposition~\ref{act2} and Theorem~\ref{Fou}, it follows that
		$$ \mathcal{F}_{b}\frac{x_{n}+D_{b,n}}{2}\mathcal{F}_{b}^{-1} = i\frac{x_{n}+D_{b,n}}{2} \hspace{30pt} \mathcal{F}_{b}\frac{x_{n}-D_{b,n}}{2}\mathcal{F}_{b}^{-1} = -i\frac{x_{n}-D_{b,n}}{2}$$
		on $W_{b,alg}$. Since $x_{n}$, $ D_{b,n}$, $\mathcal{F}_{b}$, $\mathcal{F}_{b}^{-1}$ are continuous on $W_{b,smooth}$, these identities extend to $W_{b,smooth}$ and the claim follows.
	\end{proof}
	
	\vspace{10pt}
	The Fourier-intertwining relations in Theorem~\ref{fou1} immediately imply the following basic consequences.
	
	\begin{cor}[Essential skew-adjointness of $D_{b,n}$]\label{adjoint2}
		$D_{b,n}$ is essentially skew-selfadjoint.
	\end{cor}
	\begin{proof}
		This follows from Theorem~\ref{fou1} and the essential skew-adjointness of $ix_{n}$.
	\end{proof}

	\begin{cor}\label{relation}
		Let $D_{b,n}$ be the operators defined in Definition~\ref{Db} and let $H_b$ be the deformation of Laplacian defined in Definition~\ref{Hb}. Then
		\begin{enumerate}
			\item (Commutativity of $D_{b,n}$) 
			For $m,n=1,\dots, N$, 
			$$[D_{b,m},D_{b,n}]=0.$$
			\item (Standard representation at the Lie algebra level) 
			The $\mathfrak{sl}_{2}$-triple $\left\{\frac{i}{2}|x|^2\,,\,\frac{i}{2}H_{b}\,,\,E+\frac{N+2b}{2}\right\}$ acts on the real vector spaces $\left\{x_{n}, iD_{b,n}\right\}_{\mathbb{R}}$\hspace{5pt} ($\,n=1,\dots, N\,$) as the standard representation. More specifically,
			\begin{align*}
				\MoveEqLeft\left[\frac{i}{2}|x|^2, x_{n}\right]=0	&& \left[\frac{i}{2}|x|^2,iD_{b,n}\right]= x_{n}\\
				\MoveEqLeft\left[E+\frac{N+2b}{2}, x_{n}\right]=x_{n}	&& \left[E+\frac{N+2b}{2},iD_{b,n}\right]=-iD_{b,n}\\
				\MoveEqLeft\left[\frac{i}{2}H_{b},x_{n}\right]= iD_{b,n}	&& \left[\frac{i}{2}H_{b},iD_{b,n}\right]=0.\\
			\end{align*}
			
			\item (Quadratic relations)
			\[
				|x|^2=\sum_{n=1}^N x_n^2,\qquad
				E+\frac{N+2b}{2}=\frac12\sum_{n=1}^N\{D_{b,n},x_n\},\qquad
				H_b=\sum_{n=1}^N D_{b,n}^2.
			\]
		\end{enumerate}
	\end{cor}
	\vspace{6pt}
	\begin{proof}
		All assertions follow from Theorem~\ref{fou1}, except for the identity
		$E+\frac{N+2b}{2}=\frac12\sum_{n=1}^N\{D_{b,n},x_n\}$, which we prove directly.
		\begin{enumerate}
			\item The claim is the Fourier transform of $[x_{m},x_{n}]=0$.
			\item The three equalities of the left-hand side follow from the definition. The three equalities of the right-hand side are the Fourier transform of them.
			\item The first equality is trivial, and the third one is the Fourier transform of it.
			The second one is shown as follows:
			\begin{align*}
				\MoveEqLeft 2\sum_{n=1}^{N}\left\{D_{b,n},x_{n}\right\}=\sum_{n=1}^{N}\left\{[H_{b},x_{n}],x_{n}\right\} & \\
				\MoveEqLeft = \sum_{n=1}^{N}\left\{\left(H_{b}x_{n}-x_{n}H_{b}\right)x_{n} + x_{n}\left(H_{b}x_{n}-x_{n}H_{b}\right) \right\} \\
				\MoveEqLeft = H_{b}|x|^{2}-|x|^{2}H_{b} = \left[H_{b},|x|^{2}\right] =4\left(E+\frac{N+2b}{2}\right).
			\end{align*}

		\end{enumerate}
	\end{proof}
	
	\begin{rmk}\label{weylrel}
		We note in passing that the commutator $[D_{b,m},x_n]$ can also be computed explicitly. Using the formula for $D_{b,n}$ in Theorem~\ref{explicit} in the next subsection, one obtains
		\[
		[D_{b,m},x_n]f(x)
		=
		\delta_{mn}f(x)
		+
		\frac{2b}{\mathrm{vol}(S^{N-1})}
		\int_{S^{N-1}}
		\xi_m\xi_n\,f(\sigma_\xi(x))\,d\xi
		\qquad
		(f\in \mathcal S(\mathbb R^N)),
		\]
		where $\sigma_\xi(x)=x-2\langle x,\xi\rangle \xi$ denotes the reflection. Equivalently,
		\[
		[D_{b,m},x_n]
		=
		\delta_{mn}
		+
		\frac{2b}{\mathrm{vol}(S^{N-1})}
		\int_{S^{N-1}}
		\xi_m\xi_n\,\sigma_\xi\,d\xi,
		\]
		where $(\sigma_\xi f)(x):=f(\sigma_\xi(x))$.
		When $b=0$, this reduces to the classical Weyl algebra relation
		$\left[\frac{\partial}{\partial x_m},x_n\right]=\delta_{mn}$.
	\end{rmk}
	
	\vspace{6pt}
	
	\begin{cor}[Standard representation at the group level]\label{lift2}
		Let $D_{b,n}$ be the operators defined in Definition~\ref{Db} and let $\Omega_{b}$ be the unitary representation constructed in Theorem~\ref{lift}. 
		Then $\widetilde{SL}(2,\mathbb{R})$ acts on $V_{b,n}:=\left\{x_{n}, iD_{b,n}\right\}_{\mathbb{R}}$ as the standard representation via
		$$ (g,v) \mapsto \Omega_{b}(g)\circ v\circ\Omega_{b}(g)^{-1}\hspace{24pt} (g\in\widetilde{SL}(2,\mathbb{R}), v\in V_{b,n} ).$$
	\end{cor}
	
	\vspace{3pt}
	\begin{proof}
		Let $v \in V_{b,n}$ and $X \in \mathfrak{sl}_{2}(\mathbb{R})$. Since $V_{b,n}$ is two-dimensional and stable under $\mathrm{ad}\bigl(d\Omega_b(X)\bigr)$ by Corollary~\ref{relation}, item~2, there exists $c_{0},c_{1}\in \mathbb{R}$ such that 
		\[\mathrm{ad}\bigl(d\Omega_b(X)\bigr)^{2}v=c_{1}\, \mathrm{ad}\bigl(d\Omega_b(X)\bigr)\,v \,+\, c_{0}\,v.\]
		Set $F_{k}(t):=\mathrm{ad}\bigl(d\Omega_b(X)\bigr)^{k}\,\Omega_{b}(e^{tX})\circ v\circ\Omega_{b}(e^{-tX})$. 
		Using the Fr\'{e}chet topology of $W_{b,smooth}$; see Remark~\ref{conti}, we have $\frac{dF_{k}}{dt}(t)=F_{k+1}(t)$ and $F_{2}(t)= c_{1}F_{1}(t)+c_{0}F_{0}(t)$. Thus, \[\frac{d}{dt}\begin{pmatrix}
			F_{0}(t), F_{1}(t)
		\end{pmatrix} 
		=\begin{pmatrix}
		F_{0}(t), F_{1}(t) 
		\end{pmatrix}
		\left(\begin{smallmatrix}
		0 & c_{0} \\ 1 & c_{1}
		\end{smallmatrix}\right).\] 
		Solving this system, we obtain $\begin{pmatrix}
		F_{0}(t), F_{1}(t)
		\end{pmatrix} = \begin{pmatrix}
		v, [d\Omega_b(X),v]
		\end{pmatrix}\exp(t\left(\begin{smallmatrix}
		0 & c_{0} \\ 1 & c_{1}
		\end{smallmatrix}\right))$. In particular, $F_{0}(t)=\Omega_{b}(e^{tX})\circ v\circ\Omega_{b}(e^{-tX})$ remains in $V_{b,n}$, and the induced action is given by the standard two-dimensional representation. This proves the claim.
	\end{proof}

	\vspace{24pt}
	\begin{rmk}[Order of the argument]
		We may regard Theorem~\ref{fou1} as a manifestation of the Weyl group action exchanging weights, 
		see Remark~\ref{Weyl}.
		The argument can also be reversed: once Corollaries~\ref{relation} and \ref{lift2} are established, 
		Theorem~\ref{fou1} follows from representation theory.
	\end{rmk}

	\vspace{12pt}
	
	\subsection{Explicit formula for $D_{b,n}$}\label{3.3}
	In this subsection, we derive explicit formulas for the operator $D_{b,n}$, which might be viewed as an analogue of Dunkl operators.
	
	\vspace{12pt}
	The following theorem gives two equivalent explicit expressions for $D_{b,n}$: one in terms of integration over the sphere $|x|=|y|$, and the other in terms of reflections.
	\begin{thm}[Explicit formula for $D_{b,n}$]\label{explicit} 
		Let $D_{b,n}$ be the operators defined in Definition~\ref{Db}.
		For $f\in \mathcal{S}(\mathbb{R}^{N})$, the operator $D_{b,n}$ admits the following two expressions:
		\begin{align*}
			\MoveEqLeft D_{b,n}f(x) =\frac{\partial f}{\partial x_{n}}(x) + \frac{2b}{\mathrm{vol}(S^{N-1})}\int_{|x|=|y|} \frac{x_{n}-y_{n}}{\left| x- y \right|^{N}} (f(x)-f(y))\,dy & \\
			& =\frac{\partial f}{\partial x_{n}}(x) + \frac{b}{\mathrm{vol}(S^{N-1})} \int_{S^{N-1}} \xi_{n}\frac{f(x)-f(\sigma_{\xi}(x))}{\langle \xi,x\rangle}\,d\xi, &
		\end{align*}
		
		where $d\xi$ is the $O(N)$-invariant measure on $S^{N-1}$,
		and $dy$ is the corresponding $O(N)$-invariant measure on the sphere
		$\{y\in\mathbb{R}^N : |y|=|x|\}$.
	\end{thm}
	
	\vspace{4pt}
	\begin{rmk}[Analogy with Dunkl operators]
		Let $\Phi$ be a root system, $W$ its reflection group, and $k$ a $W$-invariant function on $\Phi$.
		The Dunkl differential-difference operator~\cite{MR951883} is defined by
		\[
		T_{n}f(x)
		=
		\frac{\partial f}{\partial x_{n}}(x)
		+\frac{1}{2}\sum_{\alpha\in\Phi}
		k(\alpha)\,\alpha_{n}\,
		\frac{f(x)-f(\sigma_{\alpha}(x))}{\langle \alpha,x\rangle}.
		\]
		
		The operator $D_{b,n}$ appears to be a smooth analogue of the Dunkl operator,
		corresponding formally to the case where $W=O(N)$ and $\Phi=S^{N-1}$.
	\end{rmk}
	
	\vspace{6pt}
	
	\begin{rmk}[The case \(N=1\)]
		When $N=1$, we have $S^{N-1}=S^0=\{\pm1\}$ and $O(N)=O(1)=\{\pm1\}$, so the spherical integral in Theorem~\ref{explicit} reduces to a two-point average. Consequently,
		\[
		D_{b,1}f(x)=\frac{d}{dx}f(x)+b\,\frac{f(x)-f(-x)}{x},
		\]
		which is the Dunkl operator for the reflection group $\mathbb Z/2\mathbb Z$.
	\end{rmk}
	
	\vspace{12pt}

	\begin{proof}[Proof of Theorem~\ref*{explicit}]
		
		We first establish two auxiliary lemmas. 
		The first rewrites spherical integration in terms of reflections, and the second computes the resulting spherical integral on spherical harmonics.
		\vspace{6pt}

	\begin{lmm}\label{inte}
		For every integrable function \(f\) on \(S^{N-1}\) and for any fixed \(\omega \in S^{N-1}\),
		\[
		\int_{S^{N-1}} f(\mu)\,d\mu
		=
		\int_{S^{N-1}} f(\sigma_\xi(\omega))\,|2\langle \xi,\omega\rangle|^{N-2}\,d\xi .
		\]
		Here, \(d\mu\) and \(d\xi\) denote the \(O(N)\)-invariant measure on \(S^{N-1}\), and $\sigma_\xi(x)=x-2\langle x,\xi\rangle \xi$ denotes the reflection.
	\end{lmm}
	\vspace{0pt}
	\begin{proof}
		Fix $\omega\in S^{N-1}$ and consider the hemispheres
		$S^{N-1}_{\pm,\omega}:=\{\xi\in S^{N-1}\mid \pm\langle\xi,\omega\rangle>0\}$. Then 
		$S^{N-1}=S^{N-1}_{+,\omega}\sqcup S^{N-1}_{-,\omega}\sqcup\{\langle\xi,\omega\rangle=0\}. $
		We note that the equator has measure $0$.
		For $\varepsilon=\pm1$, $\xi\in S^{N-1}_{\varepsilon,\omega}$ can be written uniquely as
		$ \xi=\varepsilon\omega\cos\theta+\eta\sin\theta
		\,
		\left(0<\theta<\frac{\pi}{2},\,
		\eta\in S^{N-2}_\omega:=S^{N-1}\cap\omega^\perp \right). $
		In these coordinates
		$d\xi=(\sin\theta)^{N-2}d\theta d\eta$,
		$\langle\xi,\omega\rangle=\varepsilon\cos\theta$, and
		$\sigma_\xi(\omega)=-\omega\cos(2\theta)-\varepsilon\eta\sin(2\theta)$.
		
		Hence
		\begin{align*}
			\int_{S^{N-1}_{\varepsilon,\omega}}
			f(\sigma_\xi(\omega))|2\langle\xi,\omega\rangle|^{N-2}d\xi
			=
			\int_{S^{N-2}_\omega}\int_0^{\pi/2}
			f(-\omega\cos(2\theta)-\varepsilon\eta\sin(2\theta))
			(2\cos\theta)^{N-2}(\sin\theta)^{N-2}d\theta d\eta .
		\end{align*}
		
		Setting $\phi=2\theta$ yields
		\[
		\frac12
		\int_{S^{N-2}_\omega}\int_0^\pi
		f(-\omega\cos\phi-\varepsilon\eta\sin\phi)
		(\sin\phi)^{N-2}d\phi d\eta
		=
		\frac12\int_{S^{N-1}}f(\mu)d\mu .
		\]
		
		Summing over $\varepsilon=\pm1$ gives the result.
	\end{proof}

	\begin{lmm}\label{sph}
		Let $p$ be a spherical harmonic of degree $m$. Then, for $\omega\in S^{N-1}$,
		\begin{align*}
			\MoveEqLeft \frac{2}{\mathrm{vol}(S^{N-1})}\int_{S^{N-1}}
			\frac{\omega_n-\eta_n}{|\omega-\eta|^N}\bigl(p(\omega)-p(\eta)\bigr)\,d\eta \\[-5pt]
			& = \frac{1}{\mathrm{vol}(S^{N-1})}\int_{S^{N-1}}
			\xi_{n}\frac{p(\omega)-p(\sigma_{\xi}(\omega))}{\langle \omega,\xi\rangle }\,d\xi
			=
			\begin{dcases}
				0,& m=0,\\[4pt]
				\frac{1}{\lambda_{N,m}}\frac{\partial p}{\partial x_n}(\omega),& m\ge1.
			\end{dcases}
		\end{align*}
		Here, $\sigma_\xi(x)=x-2\langle x,\xi\rangle \xi$ denotes the reflection and $\lambda_{N,m}=m+\frac{N-2}{2}$.
	\end{lmm}
	\vspace{12pt}
	
	\begin{proof}
		We set $K_\nu(x,\omega):=\frac{1}{\nu}\left(|x-\omega|^{-2\nu}-1\right)$. 
		By the theory of the Poisson kernel, spherical harmonics, and Gegenbauer
		polynomials,
		\[
		\frac{1}{\mathrm{vol}(S^{N-1})}\int_{S^{N-1}}
		K_{\frac{N-2}{2}}(x,\eta)\bigl(p(\eta)-p(x)\bigr)\,d\eta
		=
		\begin{dcases}
			0,& m=0,\\[4pt]
			\frac{1}{\lambda_{N,m}}p(x),& m\ge1,
		\end{dcases}
		\]
		for $|x|\le 1$, where $\lambda_{N,m}=m+\frac{N-2}{2}$.
		Differentiating with respect to $x_n$, we obtain
		\[
		\frac{2}{\mathrm{vol}(S^{N-1})}\int_{S^{N-1}}
		\frac{\omega_n-\eta_n}{|\omega-\eta|^N}\bigl(p(\omega)-p(\eta)\bigr)\,d\eta
		=
		\begin{dcases}
			0,& m=0,\\[4pt]
			\frac{1}{\lambda_{N,m}}\frac{\partial p}{\partial x_n}(\omega),& m\ge1.
		\end{dcases}
		\]
		
		Applying Lemma~\ref{inte} to the integral with respect to $\eta$, we obtain the claim.
	\end{proof}
	
	\vspace{12pt}
	
		We now derive the explicit formula with these lemmas in hand.
		
		\vspace{6pt}
		First, we compute $D_{b,n}F(x)$$\hspace{4pt}=\hspace{4pt}$$\frac{1}{2}\left[H_{b}, x_{n}\right]F(x)$ for  functions $F(x)$ written as $F(x)= f(|x|^{2})\,p(x)$ by an $m$-th harmonic polynomial $p(x)$ and $O(N)$-invariant Schwartz function $f(|x|^{2})$. The space spanned by such functions contains $W_{b,alg}$ and is contained in $W_{b,smooth}$.
		We recall that
		\[ H_{b}=\Delta + \frac{2b}{|x|^{2}}\mathcal{R},\qquad \mathcal{R}\Bigl(f\bigl(|x|^{2}\bigr)p(x)\Bigr)= E\Bigl(f\bigl(|x|^{2}\bigr)\Bigr)p(x)\]
		
		Decomposing as
		$x_{n}\,p(x) = \left(x_{n}\,p(x)- \frac{1}{2\lambda_{N,m}}|x|^2\frac{\partial p}{\partial x_{n}}(x)\right) + \frac{1}{2\lambda_{N,m}}|x|^2\frac{\partial p}{\partial x_{n}}(x) $,
		where $x_{n}\,p(x)- \frac{1}{2\lambda_{N,m}}|x|^2\frac{\partial p}{\partial x_{n}}(x) $ is an $(m+1)$-th harmonic polynomial and 
		$\frac{\partial p}{\partial x_{n}}(x)$ is an $(m-1)$-th harmonic polynomial, we obtain
	
			\[\frac{1}{2|x|^{2}}\Bigl\lbrack \mathcal{R},x_{n} \Bigr\rbrack f\bigl(|x|^2\bigr)p(x) = \frac{1}{\lambda_{N,m}}f\bigl(|x|^{2}\bigr)\frac{\partial p}{\partial x_{n}}(x). \]
			
		Applying Lemma~\ref{sph}, this equals 
			\[\frac{2}{\mathrm{vol}(S^{N-1})}\int_{|x|=|y|}
				\frac{x_n-y_n}{|x-y|^N}\bigl(F(x)-F(y)\bigr)\,dy  = \frac{1}{\mathrm{vol}(S^{N-1})}\int_{S^{N-1}}
				\xi_{n}\frac{F(x)-F(\sigma_{\xi}(x))}{\langle x,\xi\rangle }\,d\xi. \]
		This proves the formula for $F(x) = f(|x|^{2})p(x)$. 
		Since the space spanned by such functions is dense in $W_{b,smooth}$ and the operators involved are continuous on this space (see Remark~\ref{conti}), the formula extends to $f\in \mathcal S(\mathbb R^N)$.
	\end{proof}
	
	\vspace{12pt}
	
	\subsection{Generalized translations}\label{3.4}
	In this subsection, we consider the one-parameter group generated by $D_{b,n}$, which we call the generalized translation associated with $D_{b,n}$.
	\vspace{6pt}
	
	\begin{dfn}[Generalized translations $e^{tD_{b,n}}$]\label{tra}
		By Corollary~\ref{adjoint2}, $D_{b,n}$ is essentially skew-adjoint. The corresponding unitary one-parameter group $e^{tD_{b,n}}$ is called a generalized translation.
	\end{dfn}
	
	\begin{exm}[The case \(N=1\)]
		When $N=1$, the generalized translation admits an explicit formula in terms of Legendre functions; see \cite{paper2}. 
		We note that the operator treated in \cite{paper2} is
		$|x|^b D_{b,1}|x|^{-b}$, so the normalization is slightly different from the present one.
	\end{exm}
	
	\vspace{12pt}
	We next study a basic qualitative property of this generalized translation, an analogue of finite propagation. We follow the energy method for the wave equation; see \cite[Section 2.4, Theorem 6]{MR1625845} for a reference.
	
	\vspace{6pt}
	By Theorem~\ref{explicit}, the operator $D_{b,n}$ extends naturally to $C^\infty(\mathbb R^N)$ via its explicit formula, and we shall use this realization in what follows.
	
	\begin{prp}[Finite propagation in the radial direction]\label{fin}
		Let $u(t,x)\in C^{\infty}(\mathbb{R}_{>0}\times\mathbb{R}^{N})$ satisfy the equation $\frac{\partial^2 u}{\partial t^2}(t,x) = D_{b,n}^{2}u(t,x)$. If $u(0,x)=0$ and $u_{t}(0,x)=0$ hold for $t_{0} < |x|<t_{1}$, then $u(t,x)=0$ for $t_{0}+t < |x|<t_{1}-t$.
	\end{prp}
	
	We need the following lemma.
	\begin{lmm}[A Green-type formula on balls]\label{green}
		Suppose $b>-\frac{N}{2}$. For $R>0$ and $F,G\in C^{\infty}(\mathbb R^N)$,
		\[
		\int_{B_R} (D_{b,n}F)(x)\,G(x)\,|x|^{2b}dx
		=
		-\int_{B_R} F(x)\,(D_{b,n}G)(x)\,|x|^{2b}dx
		+
		\int_{|x|=R}\frac{x_n}{|x|}F(x)G(x)\,|x|^{2b}d\omega .
		\]
		Here
		$B_R:=\{x\in \mathbb R^N:\ |x|\le R\}$
		denotes the closed ball of radius $R$.
	\end{lmm}
	
	\begin{proof}
		The proof is somewhat technical and is therefore deferred to Appendix~\ref{4.2}. 
		
		The argument uses the explicit formula in Theorem~\ref{explicit}, separating the differential part and the integral part. The differential part is handled by integration by parts, while the integral part is treated by a Fubini-type argument with some care near the singularities. The boundary term arises from the differential part.
	\end{proof}
	
	\begin{proof}[Proof of Proposition~\ref{fin}]
		Set
		$E(t):=\frac12\int_{t_0+t<|x|<t_1-t}\bigl(|u_t|^2+|D_{b,n}u|^2\bigr)\,|x|^{2b}dx$.
		Then
		\begin{align*}
			\MoveEqLeft E'(t) =
			\int_{t_0+t<|x|<t_1-t}(u_tu_{tt}+D_{b,n}u\,D_{b,n}u_t)\,|x|^{2b}dx \\
			&\quad
			-\frac12\int_{|x|=t_0+t}(|u_t|^2+|D_{b,n}u|^2)\,|x|^{2b}d\omega
			-\frac12\int_{|x|=t_1-t}(|u_t|^2+|D_{b,n}u|^2)\,|x|^{2b}d\omega .
		\end{align*}
		
		Applying Lemma~\ref{green} to the balls $B_{t_1-t}$ and $B_{t_0+t}$ 
		with
		$F=u_t(t,\cdot), G=D_{b,n}u(t,\cdot)$ and subtracting,
		we obtain
		\begin{align*}
			\MoveEqLeft\int_{t_0+t<|x|<t_1-t}D_{b,n}u\,D_{b,n}u_t\,|x|^{2b}dx \\
			& =
			-\int_{t_0+t<|x|<t_1-t}u_t\,D_{b,n}^2u\,|x|^{2b}dx \\
			&\quad
			+\int_{|x|=t_1-t}\frac{x_n}{|x|}u_t\,D_{b,n}u\,|x|^{2b}d\omega
			-\int_{|x|=t_0+t}\frac{x_n}{|x|}u_t\,D_{b,n}u\,|x|^{2b}d\omega .
		\end{align*}
		
		Hence
		\begin{align*}
			\MoveEqLeft E'(t) =
			\int_{t_0+t<|x|<t_1-t}u_t(u_{tt}-D_{b,n}^2u)\,|x|^{2b}dx \\
			&\quad
			-\int_{|x|=t_0+t}\frac{x_n}{|x|}u_t\,D_{b,n}u\,|x|^{2b}d\omega
			-\frac12\int_{|x|=t_0+t}(|u_t|^2+|D_{b,n}u|^2)\,|x|^{2b}d\omega \\
			&\quad
			+\int_{|x|=t_1-t}\frac{x_n}{|x|}u_t\,D_{b,n}u\,|x|^{2b}d\omega
			-\frac12\int_{|x|=t_1-t}(|u_t|^2+|D_{b,n}u|^2)\,|x|^{2b}d\omega .
		\end{align*}
		Since
		$u_{tt}=D_{b,n}^2u$, and $\left|\frac{x_n}{|x|}u_t\,D_{b,n}u\right|
		\le
		\frac12\bigl(|u_t|^2+|D_{b,n}u|^2\bigr)$, we get
		\[
		E'(t)\le0.
		\]
		Therefore
		\[
		0\le E(t)\le E(0)=0,
		\]
		and hence $E(t)\equiv0$. Thus \(u_t\equiv 0\) and \(D_{b,n}u\equiv 0\) in the region under consideration. Since \(u(0,x)=0\) there, it follows that \(u\equiv 0\). This proves the claim.
	\end{proof}
	
	\vspace{12pt}
	\begin{cor}[Finite propagation property for $e^{tD_{b,n}}$]\label{uni}
		Let $f, g \in \mathcal{S}(\mathbb{R}^{N})$. If $f(y)=g(y)$ for $|x|-|t|< |y| <|x|+|t|$, then 
		$$ e^{tD_{b,n}}f(x)=e^{tD_{b,n}}g(x).$$
	\end{cor}
	
	\vspace{6pt}
	\begin{proof}
		Set $h:=f-g\in \mathcal S(\mathbb R^N)$, and define
		$u(s,y):=e^{sD_{b,n}}h(y)$.
		Then $u$ satisfies
		\[
		\frac{\partial^2 u}{\partial s^2}(s,y)=D_{b,n}^2u(s,y),
		\qquad
		u(0,y)=h(y),\qquad u_s(0,y)=D_{b,n}h(y).
		\]
		By assumption,
		$h(y)=0$
		when $|x|-|t|<|y|<|x|+|t|$.
		Since this is an $O(N)$-invariant open set, Theorem~\ref{explicit} implies
		$D_{b,n}h(y)=0$ when
		$|x|-|t|<|y|<|x|+|t|$.
		Hence Proposition~\ref{fin} gives
		$u(t,x)=0$.
		Therefore
		\[
		e^{tD_{b,n}}f(x)-e^{tD_{b,n}}g(x)
		=
		e^{tD_{b,n}}h(x)
		=
		u(t,x)
		=
		0,
		\]
		which proves the claim.
	\end{proof}
	
	\vspace{12pt}
	\begin{cor}[Extension to a one-parameter group on $C^\infty(\mathbb{R}^N)$]\label{ext}
		$e^{tD_{b,n}}$ extends naturally to a one-parameter group on $C^{\infty}(\mathbb{R}^{N})$.
	\end{cor}
	
	\begin{proof}
		By Corollary~\ref{uni}, for each $x\in \mathbb R^N$ and $t\in \mathbb R$,
		the value $e^{tD_{b,n}}f(x)$ depends only on the restriction of $f$ to the annulus
		\[
			\{y\in \mathbb R^N:\ |x|-|t|<|y|<|x|+|t|\}.
		\]
		Hence, for any $f\in C^\infty(\mathbb R^N)$, we may choose
		$g\in \mathcal S(\mathbb R^N)$ such that
		$g(y)=f(y)$
		for $ |x|-|t|<|y|<|x|+|t|$,
		and define
		\[
		e^{tD_{b,n}}f(x):=e^{tD_{b,n}}g(x).
		\]
		This is well-defined by Corollary~\ref{uni}. The group property follows from that on $\mathcal S(\mathbb R^N)$.
	\end{proof}

	\vspace{14pt}

	\section{Appendix}\label{App}
	
	\subsection{Proof of Proposition~\ref*{bound2}}\label{4.1}
	
	In this appendix, we prove Proposition~\ref{bound2}, whose proof was postponed from Subsection~\ref{2.6}.
	
	\medskip
	
	{\noindent\textbf{Recall Proposition~\ref{bound2}.}\itshape\,
	Assume \(\nu>-1\), \(t\in[-1,1]\), \(w\in\mathbb C\) and \(b>-\nu-1\). Then there exist
	constants \(C_{b,\nu}>0\) and \(M_{b,\nu}\ge0\) such that
	\[
	\bigl|\mathscr I_{b,\nu}(w,t)\bigr|
	\le
	C_{b,\nu}(1+|w|)^{M_{b,\nu}}e^{|\mathrm{Re}\,w|}.
	\]
	}
	
	\vspace{6pt}
	
	First we estimate the Bessel function. Throughout, we write
	\[
	\widetilde{I}_b(z):=\frac{\mathcal I_b(z)}{\Gamma(b+1)}
	=\sum_{m=0}^\infty \frac{(z/2)^{2m}}{\Gamma(b+m+1)m!}
	\]
	for a normalized \(I\)-Bessel function. The function \(\widetilde{I}_b(z)\) is entire in \(b\), and
	$\frac{d}{dz}\widetilde{I}_b(z)=\frac{z}{2}\widetilde{I}_{b+1}(z)$.
	
	\vspace{12pt}
	\begin{lmm}\label{lmm:Bessel-bound}
		For every fixed \(b\in\mathbb R\), there exist constants \(C_b>0\) and
		\(M_b\ge0\) such that
		\[
		|\widetilde{I}_b(z)|\le C_b(1+|z|)^{M_b}e^{|\mathrm{Re}\,z|}
		\qquad (z\in\mathbb C).
		\]
	\end{lmm}

	\begin{proof}
		Choose \(m\in\mathbb N\) so that \(b+m>-\frac12\). For such indices, the
		integral representation
		$\widetilde{I}_{b+m}(z)
		=
		\frac{1}{\Gamma\!\left(b+m+\frac12\right)\Gamma\left(\frac12\right)}
		\int_{-1}^1 e^{zs}(1-s^2)^{b+m-\frac12}\,ds$
		gives $	|\widetilde{I}_{b+m}(z)|\le \frac{1}{\Gamma(b+m+1)} e^{|\mathrm{Re}\,z|}$ 
		
		\vspace{6pt}
		Now set
		\[
		\Psi_b(\zeta):=\widetilde{I}_b(\zeta z),
		\qquad 0\le \zeta\le 1.
		\]
		Using
		$\Psi_b'(\zeta)=\frac{\zeta z^2}{2}\widetilde{I}_{b+1}(\zeta z)$ 
		we obtain
		\[
		\widetilde{I}_b(z)-\widetilde{I}_b(0)
		=
		\frac{z^2}{2}\int_0^1 \zeta\,\widetilde{I}_{b+1}(\zeta z)\,d\zeta.
		\]
		Iterating this identity finitely many times reduces the estimate for
		\(\widetilde{I}_b\) to that for \(\widetilde{I}_{b+m}\), and each step
		introduces only a polynomial factor in \(|z|\). This proves the claim.
	\end{proof}
	
		\vspace{12pt}
	\begin{proof}[Proof of Proposition~\ref*{bound2}]
	
		Assume first that \(b>0\). Set $F_b(u;w,t):=\widetilde{I}_b(uw)e^{(1-u)wt},
		\, (0\le u\le 1)$. \\
		Since
		$\frac{1}{B(b,\nu+1)}\mathcal I_b(uw)e^{(1-u)wt}
		=
		\frac{b\,\Gamma(b+\nu+1)}{\Gamma(\nu+1)}\,F_b(u;w,t)$,
		we may write
		\[
		\mathscr I_{b,\nu}(w,t)
		=
		\frac{b\,\Gamma(b+\nu+1)}{\Gamma(\nu+1)}
		\int_0^1 u^{b-1}(1-u)^\nu F_b(u;w,t)\,du.
		\]
		
		Let $m\in\mathbb{N}$. We expand \(F_b\) at \(u=0\) in the form
		\[
		F_b(u;w,t)
		=
		\sum_{k=0}^{m-1}\frac{F_b^{(k)}(0;w,t)}{k!}\,u^k
		+
		u^m\int_{\Delta_m}F_b^{(m)}(u\tau_m;w,t)\,d\tau,
		\]
		where 
		$\Delta_m
		:=
		\{(\tau_1,\dots,\tau_m)\in[0,1]^m:\ 0\le\tau_1\le\cdots\le\tau_m\le1\}$.
		Substituting this into the integral, we obtain
		\begin{equation}\label{an}
		\mathscr I_{b,\nu}(w,t)
		=
		\frac{b\,\Gamma(b+\nu+1)}{\Gamma(\nu+1)}
		\sum_{k=0}^{m-1}\frac{F_b^{(k)}(0;w,t)}{k!}B(b+k,\nu+1)
		+
		R_{b,\nu}^{(m)}(w,t),
		\end{equation}
		where
		\[
		R_{b,\nu}^{(m)}(w,t)
		=
		\frac{b\,\Gamma(b+\nu+1)}{\Gamma(\nu+1)}
		\int_0^1 u^{b+m-1}(1-u)^\nu
		\int_{\Delta_m}F_b^{(m)}(u\tau_m;w,t)\,d\tau\,du.
		\]
		
		\vspace{12pt}
		Now fix $b$ satisfying \(\mathrm{Re}(b)>-\nu-1\) and choose an integer \(m\in\mathbb{N}\) such that
		$\mathrm{Re}(b)+m>0$, and consider the right-hand side of the above formula.
		\vspace{6pt}
		
		We first compute the principal part. \\
		Since 
		$\widetilde{I}_b(uw)=
		\sum_{r=0}^{\infty}
		\frac{(uw/2)^{2r}}{\Gamma(b+r+1)r!}$, and
		$e^{(1-u)wt}
		=
		e^{wt}\sum_{j=0}^{\infty}\frac{(-uwt)^j}{j!}$,
		we have
		\[
		F_b(u;w,t)
		=
		e^{wt}\sum_{r,j\ge0}
		\frac{(w/2)^{2r}(-wt)^j}{\Gamma(b+r+1)r!j!}\,u^{2r+j}.
		\]
		Hence
		\[
		\frac{F_b^{(k)}(0;w,t)}{k!}
		=
		e^{wt}\sum_{2r+j=k}
		\frac{(w/2)^{2r}(-wt)^j}{\Gamma(b+r+1)r!j!}.
		\]
		Therefore
		\[
		\frac{b\,\Gamma(b+\nu+1)}{\Gamma(\nu+1)}
		\frac{F_b^{(k)}(0;w,t)}{k!}B(b+k,\nu+1)
		=
		e^{wt}\sum_{2r+j=k}
		\frac{(w/2)^{2r}(-wt)^j}{r!j!}\,C_{r,j}(b,\nu),
		\]
		where
		\[
		C_{r,j}(b,\nu)
		:=
		\frac{b\,\Gamma(b+k)}{\Gamma(b+r+1)}\,
		\frac{1}{(b+\nu+1)_k},
		\qquad k=2r+j.
		\]
		Thus every term in the principal
		part is a polynomial-exponential term whose coefficient is holomorphic for \(\mathrm{Re}(b)>-\nu-1\). In addition,
		\[
		\left|
		\frac{b\,\Gamma(b+\nu+1)}{\Gamma(\nu+1)}
		\frac{F_b^{(k)}(0;w,t)}{k!}B(b+k,\nu+1)
		\right|
		\le
		C_{b,\nu,k}(1+|w|)^k e^{|\mathrm{Re}\,w|}.
		\]

		\vspace{24pt}
		We next consider the remainder term. By Leibniz' rule,
		\[
		F_b^{(m)}(u;w,t)
		=
		\sum_{\ell=0}^{m}\binom{m}{\ell}
		\bigl(\partial_u^\ell \widetilde{I}_b(uw)\bigr)
		\bigl(\partial_u^{m-\ell}e^{(1-u)wt}\bigr).
		\]
		Since $\frac{d}{du}\widetilde{I}_\beta(uw)=\frac{uw^2}{2}\widetilde{I}_{\beta+1}(uw)$,
		\(F_b^{(m)}\) is a finite sum of terms of the form
		\[
		P_{\ell,j}(u,t)\,w^{\ell+j}\widetilde{I}_{b+j}(uw)e^{(1-u)wt},
		\qquad 0\le j\le \ell\le m,
		\]
		where \(P_{\ell,j}(u,t)\) is a polynomial in \(u\) and \(t\). By
		Lemma~\ref{lmm:Bessel-bound},
		\begin{align*}
			\MoveEqLeft |F_b^{(m)}(u;w,t)|
			\le
			C_{b,\nu,m}(1+|w|)^{M_{b,\nu,m}}
			e^{u|\mathrm{Re}\,w|}\,|e^{(1-u)wt}| \\
			&  \le
			C_{b,\nu,m}(1+|w|)^{M_{b,\nu,m}}
			e^{|\mathrm{Re}\,w|}
		\end{align*}
		
		For fixed \(u\in[0,1]\), the integrand defining \(R_{b,\nu}^{(m)}(w,t)\) is
		holomorphic in \(b\), and the above bound is locally uniform in \(b\) on
		compact subsets of \(\{b\in\mathbb C:\operatorname{Re}(b)>-\nu-1\}\).
		Hence the remainder term extends holomorphically to \(\mathrm{Re}(b)>-\nu-1\).
		
		In addition, 
		\[
		|R_{b,\nu}^{(m)}(w,t)|
		\le
		\frac{|b|\,\Gamma(b+\nu+1)}{m!\,\Gamma(\nu+1)}
		C_{b,\nu,m}(1+|w|)^{M_{b,\nu,m}}e^{|\mathrm{Re}\,w|}
		\int_0^1 u^{\mathrm{Re}(b)+m-1}(1-u)^\nu\,du.
		\]
		Because \(\mathrm{Re}(b)+m>0\) and \(\nu>-1\), the last integral equals \(B(\mathrm{Re}(b)+m,\nu+1)\)
		and is finite. Hence
		\[
		|R_{b,\nu}^{(m)}(w,t)|
		\le
		C''_{b,\nu}(1+|w|)^{M''_{b,\nu}}e^{|\mathrm{Re}\,w|}.
		\]
		
		\vspace{12pt}
		
		Since both the
		principal part and the remainder term extend holomorphically to \(\mathrm{Re}(b)>-\nu-1\),
		\eqref{an} extends to  \(\mathrm{Re}(b)>-\nu-1\) by analytic continuation. The above
		estimates for these two terms therefore yield
		\[
		|\mathscr I_{b,\nu}(w,t)|
		\le
		C_{b,\nu}(1+|w|)^{M_{b,\nu}}e^{|\mathrm{Re}\,w|}.
		\]
	\end{proof}
	
	\subsection{Proof of Lemma~\ref*{green}}\label{4.2}
	
	In this appendix, we prove Lemma~\ref{green}, whose proof was postponed from Subsection~\ref{3.4}.
	
	\medskip
	
	{\noindent\textbf{Recall Lemma~\ref*{green}.}\itshape\,
		Suppose $b>-\frac{N}{2}$. For $R>0$ and $F,G\in C^{\infty}(\mathbb R^N)$,
		\[
		\int_{B_R} (D_{b,n}F)(x)\,G(x)\,|x|^{2b}dx
		=
		-\int_{B_R} F(x)\,(D_{b,n}G)(x)\,|x|^{2b}dx
		+
		\int_{|x|=R}\frac{x_n}{|x|}F(x)G(x)\,|x|^{2b}d\omega .
		\]
	}
	
	We first prove an integral formula on the sphere. 
	\begin{lmm}\label{sphereint}
		For $x\in \mathbb R^N\setminus\{0\}$, $\varepsilon>0$, and $n=1,\dots,N$,
		\[
		\int_{\{\xi\in S^{N-1}:|\langle \xi,x\rangle|>\varepsilon\}}
		\frac{\xi_n}{\langle \xi,x\rangle}\,d\xi
		=
		\frac{x_n}{|x|^2}\,
		\mathrm{vol}\!\left(\{\xi\in S^{N-1}:|\langle \xi,x\rangle|>\varepsilon\}\right).
		\]
	\end{lmm}
	
	\begin{proof}
		Set $\eta:=x/|x|$, and let $\sigma_\eta$ be the reflection with respect to the hyperplane $\eta^\perp$.
		Since $\sigma_\eta$ preserves the set
		\[
		\{\xi\in S^{N-1}:|\langle \xi,x\rangle|>\varepsilon\},
		\]
		and
		\[
		\langle \sigma_\eta(\xi),x\rangle=-\langle \xi,x\rangle,
		\qquad
		\sigma_\eta(\xi)_n=\xi_n-2\langle \xi,\eta\rangle \eta_n,
		\]
		we have
		\begin{align*}
			\int_{|\langle \xi,x\rangle|>\varepsilon}\frac{\xi_n}{\langle \xi,x\rangle}\,d\xi
			&=
			\int_{|\langle \xi,x\rangle|>\varepsilon}
			\frac{\sigma_\eta(\xi)_n}{\langle \sigma_\eta(\xi),x\rangle}\,d\xi \\
			&=
			-\int_{|\langle \xi,x\rangle|>\varepsilon}\frac{\xi_n}{\langle \xi,x\rangle}\,d\xi
			+
			2\frac{x_n}{|x|^2}
			\int_{|\langle \xi,x\rangle|>\varepsilon}d\xi.
		\end{align*}
		This proves the claim.
	\end{proof}
	
	\begin{proof}[Proof of Lemma~\ref*{green}]
		For $0<\delta<R$, set
		$B_{R,\delta}:=\{x\in\mathbb R^N:\delta<|x|<R\}$.
		
		By Theorem~\ref{explicit},
		\begin{align*}
			\MoveEqLeft \int_{B_{R,\delta}} (D_{b,n}F)(x)\,G(x)\,|x|^{2b}dx \\
			&=
			\int_{B_{R,\delta}}
			\left(
			\frac{\partial F}{\partial x_n}(x)
			+
			\frac{b}{\mathrm{vol}(S^{N-1})}
			\int_{S^{N-1}}
			\xi_n\frac{F(x)-F(\sigma_\xi(x))}{\langle \xi,x\rangle}\,d\xi
			\right)
			G(x)\,|x|^{2b}dx .
		\end{align*}
		
		We first treat the differential part. By integration by parts on $B_{R,\delta}$,
		\begin{align}
			\MoveEqLeft \int_{B_{R,\delta}} \frac{\partial F}{\partial x_n}(x)\,G(x)\,|x|^{2b}dx \nonumber\\
			&=
			-\int_{B_{R,\delta}} F(x)\frac{\partial G}{\partial x_n}(x)\,|x|^{2b}dx
			-2b\int_{B_{R,\delta}} F(x)G(x)\frac{x_n}{|x|^2}\,|x|^{2b}dx \nonumber\\
			&\quad
			+\int_{|x|=R}\frac{x_n}{|x|}F(x)G(x)\,|x|^{2b}d\omega
			-\int_{|x|=\delta}\frac{x_n}{|x|}F(x)G(x)\,|x|^{2b}d\omega .
			\label{green-partial}
		\end{align}
		
		Next we treat the non-local part. We have
		\begin{align*}
			&\int_{B_{R,\delta}}
			\left(
			\int_{S^{N-1}}
			\xi_n\frac{F(x)-F(\sigma_\xi(x))}{\langle \xi,x\rangle}\,d\xi
			\right)
			G(x)\,|x|^{2b}dx \\
			&=
			\lim_{\varepsilon\to+0}
			\int_{B_{R,\delta}}
			\left(
			\int_{\{\xi\in S^{N-1}:|\langle \xi,x\rangle|>\varepsilon\}}
			\xi_n\frac{F(x)-F(\sigma_\xi(x))}{\langle \xi,x\rangle}\,d\xi
			\right)
			G(x)\,|x|^{2b}dx .
		\end{align*}
		Since $\sigma_\xi$ preserves $B_{R,\delta}$, Lebesgue measure, and $|x|$, and satisfies
		\[
		\langle \xi,\sigma_\xi(x)\rangle=-\langle \xi,x\rangle,
		\qquad
		\sigma_\xi(\sigma_\xi(x))=x,
		\]
		we obtain
		\begin{align*}
			&\int_{B_{R,\delta}}
			\left(
			\int_{S^{N-1}}
			\xi_n\frac{F(x)-F(\sigma_\xi(x))}{\langle \xi,x\rangle}\,d\xi
			\right)
			G(x)\,|x|^{2b}dx \\
			&=
			\lim_{\varepsilon\to+0}
			\int_{B_{R,\delta}}
			F(x)
			\left(
			\int_{\{\xi\in S^{N-1}:|\langle \xi,x\rangle|>\varepsilon\}}
			\xi_n\frac{G(x)+G(\sigma_\xi(x))}{\langle \xi,x\rangle}\,d\xi
			\right)
			|x|^{2b}dx \\
			&=
			\lim_{\varepsilon\to+0}
			2\int_{B_{R,\delta}}
			F(x)G(x)
			\left(
			\int_{\{\xi\in S^{N-1}:|\langle \xi,x\rangle|>\varepsilon\}}
			\frac{\xi_n}{\langle \xi,x\rangle}\,d\xi
			\right)
			|x|^{2b}dx \\
			&\hspace{36pt}
			-
			\int_{B_{R,\delta}}
			F(x)
			\left(
			\int_{S^{N-1}}
			\xi_n\frac{G(x)-G(\sigma_\xi(x))}{\langle \xi,x\rangle}\,d\xi
			\right)
			|x|^{2b}dx .
		\end{align*}
		Here the difference quotient
		$\frac{G(x)-G(\sigma_\xi(x))}{\langle \xi,x\rangle}$
		extends smoothly in $x$, and hence the corresponding integral term admits the limit $\varepsilon\to0$ without difficulty.
		
		By Lemma~\ref{sphereint},
		\[
		\int_{\{\xi\in S^{N-1}:|\langle \xi,x\rangle|>\varepsilon\}}
		\frac{\xi_n}{\langle \xi,x\rangle}\,d\xi
		=
		\frac{x_n}{|x|^2}
		\mathrm{vol}\!\left(\{\xi\in S^{N-1}:|\langle \xi,x\rangle|>\varepsilon\}\right).
		\]
		Hence
		\begin{align*}
			&\int_{B_{R,\delta}}
			\left(
			\int_{S^{N-1}}
			\xi_n\frac{F(x)-F(\sigma_\xi(x))}{\langle \xi,x\rangle}\,d\xi
			\right)
			G(x)\,|x|^{2b}dx \\
			&=
			\lim_{\varepsilon\to+0}
			2\int_{B_{R,\delta}}
			F(x)G(x)\frac{x_n}{|x|^2}
			\mathrm{vol}\!\left(\{\xi\in S^{N-1}:|\langle \xi,x\rangle|>\varepsilon\}\right)
			|x|^{2b}dx \\
			&\hspace{36pt}
			-
			\int_{B_{R,\delta}}
			F(x)
			\left(
			\int_{S^{N-1}}
			\xi_n\frac{G(x)-G(\sigma_\xi(x))}{\langle \xi,x\rangle}\,d\xi
			\right)
			|x|^{2b}dx .
		\end{align*}
		
		Letting $\varepsilon\to0$, we obtain
		\begin{align}
			&\int_{B_{R,\delta}}
			\left(
			\int_{S^{N-1}}
			\xi_n\frac{F(x)-F(\sigma_\xi(x))}{\langle \xi,x\rangle}\,d\xi
			\right)
			G(x)\,|x|^{2b}dx \nonumber\\
			&=
			2\,\mathrm{vol}(S^{N-1})
			\int_{B_{R,\delta}} F(x)G(x)\frac{x_n}{|x|^2}\,|x|^{2b}dx \nonumber\\
			&\hspace{36pt}
			-
			\int_{B_{R,\delta}}
			F(x)
			\left(
			\int_{S^{N-1}}
			\xi_n\frac{G(x)-G(\sigma_\xi(x))}{\langle \xi,x\rangle}\,d\xi
			\right)
			|x|^{2b}dx .
			\label{green-nonlocal-pre}
		\end{align}

		By \eqref{green-partial} and \eqref{green-nonlocal-pre}, we obtain
		\begin{align}
			\int_{B_{R,\delta}} (D_{b,n}F)(x)\,G(x)\,|x|^{2b}dx
			&=
			-\int_{B_{R,\delta}} F(x)\,(D_{b,n}G)(x)\,|x|^{2b}dx \nonumber\\
			&\quad
			+\int_{|x|=R}\frac{x_n}{|x|}F(x)G(x)\,|x|^{2b}d\omega
			-\int_{|x|=\delta}\frac{x_n}{|x|}F(x)G(x)\,|x|^{2b}d\omega .
			\label{green-delta}
		\end{align}
		
		To justify the limit $\delta\to0$, note that
		\[
		F(x)G(x)=F(0)G(0)+O(|x|)
		\qquad (x\to0).
		\]
		Hence
		\begin{align*}
			\MoveEqLeft \int_{|x|=\delta}\frac{x_n}{|x|}F(x)G(x)\,|x|^{2b}d\omega \\
			&=
			\int_{|x|=\delta}\frac{x_n}{|x|}
			\bigl(F(0)G(0)+O(|x|)\bigr)\,|x|^{2b}d\omega = O(\delta^{N+2b}),
		\end{align*}
		since
		\[
		\int_{|x|=\delta}\frac{x_n}{|x|}\,d\omega=0.
		\]
		by symmetry.
		
		For the bulk terms, since $D_{b,n}F,D_{b,n}G\in C^\infty(\mathbb R^N)$, we have
		\[
		\int_{B_\delta}(D_{b,n}F)(x)G(x)\,|x|^{2b}dx = O(\delta^{N+2b}),
		\qquad
		\int_{B_\delta}F(x)(D_{b,n}G)(x)\,|x|^{2b}dx= O(\delta^{N+2b})
		\qquad (\delta\to0),
		\]
		
		Therefore
		\[
		\int_{B_R} (D_{b,n}F)(x)\,G(x)\,|x|^{2b}dx
		=
		-\int_{B_R} F(x)\,(D_{b,n}G)(x)\,|x|^{2b}dx
		+
		\int_{|x|=R}\frac{x_n}{|x|}F(x)G(x)\,|x|^{2b}d\omega .
		\]
		This proves the lemma.
	\end{proof}

	\section*{Acknowledgements}
	The author would like to express his gratitude to his supervisor, Professor Toshiyuki Kobayashi, for his continuous support and encouragement. 
	This research was supported partially by JSPS KAKENHI Grant Number JP24KJ0937 and Forefront Physics and Mathematics Program to Drive Transformation (FoPM), a World-leading Innovative Graduate Study (WINGS) Program, The University of Tokyo.

	\bibliography{paper3}	
	\bibliographystyle{alpha} 

\end{document}